\documentclass{amsart}

\usepackage{amsmath,amssymb}
\usepackage{wrapfig}
\usepackage{tikz,color}
\usepackage{graphics}
\usetikzlibrary{snakes,shapes}

\def\spc{\enspace}

\def\C{\mathbb{C}}
\def\N{\mathbb{N}}
\def\bfa{\mathbf{a}}
\def\bfb{\mathbf{b}}
\def\bfc{\mathbf{c}}
\def\bse{\mathbf{\hat e}}
\def\bfu{\mathbf{u}}
\def\bfv{\mathbf{v}}
\def\bfw{\mathbf{w}}
\def\bfx{\mathbf{x}}
\def\bfy{\mathbf{y}}

\def\tensor{\otimes}
\def\dtensor{\tensor\cdots\tensor}

\def\Id{\mathsf{Id}}
\def\tr{\mathsf{tr}}
\def\det{\mathsf{det}}
\def\adj{\mathsf{adj}}
\def\sgn{\mathrm{sgn}}

\def\mnn{M_{n\times n}}

\def\vprod#1{V^{\tensor #1}}

\def\pmx#1#2{\begin{pmatrix}#1\\#2\end{pmatrix}}                      
\def\tmx#1#2{\begin{bmatrix}#1\\#2\end{bmatrix}}                      
\def\tmxt#1#2{\bigl[\begin{smallmatrix}#1\\#2\end{smallmatrix}\bigr]} 

\newtheorem{prop}{Proposition}[section]
\newtheorem{thm}[prop]{Theorem}
\newtheorem{lemma}[prop]{Lemma}

\theoremstyle{definition}
\newtheorem{defn}[prop]{Definition}

\theoremstyle{remark}
\newtheorem{remark}[prop]{Remark}
\newtheorem{example}[prop]{Example}
\newtheorem{exercise}[prop]{Exercise}
\newtheorem*{soln}{Solution}

\begin{document}

%
%
%
%
%
%

\tikzstyle heightone=[scale=.7,xscale=.6,shift={(0,-.4)}]
\tikzstyle{every picture}=[baseline=0pt,heightone]

\tikzstyle heightones=[scale=.8,xscale=.35,shift={(0,.1)}]
\tikzstyle heightoneonehalf=[scale=.9,shift={(0,-.2)}]
\tikzstyle heighttwo=[scale=.8,shift={(0,-.4)}]
\tikzstyle heighttwos=[scale=.5,xscale=.6,shift={(0,-.1)}]
\tikzstyle heightthree=[scale=.6,shift={(0,-.9)}]
\tikzstyle heightthrees=[scale=.4,xscale=.7,shift={(0,-.2)}]

\tikzstyle nlabel=[draw=none,fill=none,shape=rectangle,inner sep=2pt,scale=.8]
\tikzstyle slabel=[draw=none,fill=none,shape=rectangle,inner sep=2pt,scale=.6]

\tikzstyle trivalent=[very thick]
\tikzstyle nodedot=[shape=circle,fill=black,draw,minimum size=2pt,inner sep=0pt]
\tikzstyle dotedge=[nodedot,pos=.5]
\tikzstyle midedge=[coordinate,pos=.5]

\tikzstyle matrix=[semithick,shape=ellipse,draw,fill=white,scale=.8,inner sep=1pt]
\tikzstyle smatrix=[semithick,shape=ellipse,draw,fill=white,scale=.6,inner sep=1pt]
\tikzstyle midmx=[matrix,pos=.5]
\tikzstyle midsmx=[smatrix,pos=.5]

\tikzstyle vector=[thin,rectangle,fill=white,scale=.8,inner sep=2pt]
\tikzstyle svector=[thin,rectangle,fill=white,scale=.7,inner sep=2pt]

\tikzstyle symlabel=[draw=none,fill=none,white,scale=.8]
\tikzstyle asymlabel=[draw=none,fill=none,black,scale=.8]
\tikzstyle symlabelleft=[nlabel,left]
\tikzstyle symlabelright=[nlabel,right]

\def\cilson{\tikzstyle ciliated=[purple]}
\def\cilsoff{\tikzstyle ciliated=[draw=none]}
\cilson
\def\arrowson{\tikzstyle oriented=[-stealth,blue]}
\def\arrowsoff{\tikzstyle oriented=[-,black]}
\arrowson
\tikzstyle nullifier=[thick,snake=coil,segment amplitude=.5pt,segment length=2pt,segment aspect=0]


\def\iml(#1){\tikz{\draw(0,-.2)..controls+(.05,.5)and+(-.05,-.5)..node[midsmx]{$#1$}(0,1.2);}}
\def\ims(#1){\tikz{\draw(0,0)..controls+(.05,.5)and+(-.05,-.5)..node[midsmx]{$#1$}(0,1);}}

\def\ima(#1){\tikz{\hpathmxa(0,0,0)(0,1,0)(#1)}}
\def\imar(#1){\tikz{\hpathmxar(0,0,0)(0,1,0)(#1)}}

\def\imm(#1,#2){\tikz[heighttwo]{\draw(0,0)--(0,.2)..controls+(.05,.3)and+(-.05,-.3)..node[midsmx]{$#2$}(0,1)
	..controls+(.05,.3)and+(-.05,-.3)..node[midsmx]{$#1$}(0,1.8)--(0,2);}}

\def\immm(#1,#2,#3){\tikz[heightthree]{\draw(0,0)--(0,.2)..controls+(.05,.3)and+(-.05,-.3)..node[midsmx]{$#3$}(0,1)
	..controls+(.05,.3)and+(-.05,-.3)..node[midsmx]{$#2$}(0,2)..controls+(.05,.3)and+(-.05,-.3)..node[midsmx]{$#1$}(0,2.8)--(0,3);}}

\def\ivup(#1){\tikz{\draw(0,0)..controls+(.05,.5)and+(-.05,-.5)..(0,1)node[vector]{$#1$};}}
\def\iupv(#1){\tikz{\draw(0,0)node[vector]{$#1$}..controls+(.05,.5)and+(-.05,-.5)..(0,1);}}

\def\ivv(#1,#2){\tikz{\draw(0,0)node[vector]{$#2$}..controls+(.05,.5)and+(-.05,-.5)..(0,1)node[vector]{$#1$};}}

\def\ivm(#1,#2){\tikz{\draw(0,0)..controls+(.05,.5)and+(-.05,-.5)..node[midsmx]{$#2$}(0,1)node[svector]{$#1$};}}
\def\imv(#1,#2){\tikz{\draw(0,0)node[svector]{$#2$}..controls+(.05,.5)and+(-.05,-.5)..node[midsmx]{$#1$}(0,1);}}

\def\ivmv(#1,#2,#3){\tikz{
	\draw(0,-.1)node[svector]{$#3$}..controls+(.05,.5)and+(-.05,-.5)..node[midsmx]{$#2$}(0,1.1)node[svector]{$#1$};
}}


\def\cupml(#1){\tikz{
		\draw(0,1)..controls+(-.03,-.3)and+(.03,.3)..node[midsmx]{$#1$}(0,.3)
			..controls+(-.05,-.5)and+(-.05,-.5)..(1,.3)..controls+(.03,.3)and+(-.03,-.3)..(1,1);}}
\def\cupmr(#1){\tikz{
		\draw(0,1)..controls+(-.03,-.3)and+(.03,.3)..(0,.3)
			..controls+(-.05,-.5)and+(-.05,-.5)..(1,.3)..controls+(.03,.3)and+(-.03,-.3)..node[midsmx]{$#1$}(1,1);}}
\def\cupmm(#1,#2){\tikz{
		\draw(0,1)..controls+(-.03,-.3)and+(.03,.3)..node[midsmx]{$#1$}(0,.3)
			..controls+(-.05,-.5)and+(-.05,-.5)..(1,.3)..controls+(.03,.3)and+(-.03,-.3)..node[midsmx]{$#2$}(1,1);}}

\def\capml(#1){\tikz{
		\draw(0,0)..controls+(.03,.3)and+(-.03,-.3)..node[midsmx]{$#1$}(0,.7)
			..controls+(.05,.5)and+(.05,.5)..(1,.7)..controls+(-.03,-.3)and+(.03,.3)..(1,0);}}
\def\capmr(#1){\tikz{
		\draw(0,0)..controls+(.03,.3)and+(-.03,-.3)..(0,.7)
			..controls+(.05,.5)and+(.05,.5)..(1,.7)..controls+(-.03,-.3)and+(.03,.3)..node[midsmx]{$#1$}(1,0);}}
\def\capmm(#1,#2){\tikz{
		\draw(0,0)..controls+(.03,.3)and+(-.03,-.3)..node[midsmx]{$#1$}(0,.7)
			..controls+(.05,.5)and+(.05,.5)..(1,.7)..controls+(-.03,-.3)and+(.03,.3)..node[midsmx]{$#2$}(1,0);}}

\def\circml(#1){\tikz{
		\draw(0,.7)..controls+(-.03,-.3)and+(.03,.3)..node[midsmx]{$#1$}(0,.3)
			..controls+(-.05,-.5)and+(-.06,-.6)..(1,.5)..controls+(.06,.6)and+(.05,.5)..(0,.7);}}
\def\circmll(#1){\tikz{
		\draw(1.2,.5)..controls+(-.08,-.8)and+(-.07,-.7)..(-.2,.4)..controls+(.03,.3)and+(-.03,-.3)..node[midmx]{$#1$}(-.2,.6)
			..controls+(.07,.7)and+(.08,.8)..(1.2,.5);}}
\def\circmr(#1){\tikz{
		\draw(0,.5)..controls+(-.06,-.6)and+(-.05,-.5)..(1,.3)..controls+(.03,.3)and+(-.03,-.3)..node[midsmx]{$#1$}(1,.7)
			..controls+(.05,.5)and+(.06,.6)..(0,.5);}}
\def\circmrl(#1){\tikz{
		\draw(-.2,.5)..controls+(-.08,-.8)and+(-.07,-.7)..(1.2,.4)..controls+(.03,.3)and+(-.03,-.3)..node[midmx]{$#1$}(1.2,.6)
			..controls+(.07,.7)and+(.08,.8)..(-.2,.5);}}
\def\circmm(#1,#2){\tikz{
		\draw(0,.7)..controls+(-.03,-.3)and+(.03,.3)..node[midsmx]{$#1$}(0,.3)
			..controls+(-.05,-.5)and+(-.05,-.5)..(1,.3)..controls+(.03,.3)and+(-.03,-.3)..node[midsmx]{$#2$}(1,.7)
			..controls+(.05,.5)and+(.05,.5)..(0,.7);}}


\def\tpermxel(#1){
	\tikz[heighttwo]{
		\foreach\xa/\xb/\xc in{#1}{
			\draw(\xa,0)node[svector]{$\xa$}..controls+(.05,.5)and+(-.05,-.5)..node[midsmx,pos=.9]{$\xb$}(\xa,1)
				..controls+(.05,.5)and+(-.05,-.5)..(\xc,2)node[svector]{$\xc$};}
	}}


\def\mupn(#1,#2){\tikz[xscale=.8]{
		\foreach\xa in{0,.5,2}{\draw(\xa,0)..controls+(.05,.5)and+(-.05,-.5)..node[midsmx]{$#1$}(\xa,1);}
		\node[nlabel,right]at(2,1){$#2$};
		\foreach\xa/\xb in{1/.2,1.5/.2,1/.8,1.5/.8}{\node[nlabel,scale=1.5]at(\xa,\xb){.};}
}}

\def\mupupn#1(#2){\tikz[heightthree]{
		\foreach\xa/\xb in{#2}{
			\draw(0,\xa)..controls+(.03,.3)and+(-.03,-.3)..node[midsmx]{$\xb$}([shift={(0,1)}]0,\xa);
		}
	}
}

\def\circnl#1#2(#3)(#4){\tikz{%
	\draw(0,0)\foreach\xa/\xb in{#3}{--(0,\xa)node[smatrix]{\xb}}--(0,#2)..controls+(0,.8)and+(0,.8)..(#1,#2)%
		\foreach\xa/\xb in{#4}{--(#1,\xa)node[smatrix]{\xb}}--(#1,0)..controls+(0,-.8)and+(0,-.8)..(0,0);%
}}

\def\nodendownmx(#1,#2){\tikz{
	\cilur<(0,1)>(-.5,.2)
	\foreach\xa in{-1.5,-.75,1.5}{\hpathmxa(\xa,0,0)(0,1,\xa)(#2)}
	\foreach\xa in{0,.75}{\hpathph(\xa,0,0)(0,1,\xa)}
	\draw(.2,.25)node[slabel]{$#1$};
}}	

\def\asymnkmx(#1,#2,#3){\tikz[heighttwo,xscale=.7]{
	\cilr<(0,1)>(-.8)
	\foreach\xa in{-2,-1,2}{\hpathmxa(0,1,\xa)(\xa,0,0)(#3)}
	\foreach\xa in{0,1}{\hpathph(0,1,\xa)(\xa,0,0)}
	\foreach\xa in{-2,1,2}{\hpatha(0,1,\xa)(\xa,2,0)}
	\foreach\xa in{-1,0}{\hpathph(0,1,\xa)(\xa,2,0)}
	\draw(-.3,1.75)node[slabel]{$#1$};
	\draw(.3,.25)node[slabel]{$#2$};
}}
\def\asymnkmxB(#1,#2,#3){\tikz[heighttwo,xscale=.7]{
	\cilr<(0,1)>(-.8)
	\foreach\xa in{-2,-1,2}{\hpatha(0,1,\xa)(\xa,0,0)}
	\foreach\xa in{0,1}{\hpathph(0,1,\xa)(\xa,0,0)}
	\foreach\xa in{-2,1,2}{\hpathmxa(0,1,\xa)(\xa,2,0)(#3)}
	\foreach\xa in{-1,0}{\hpathph(0,1,\xa)(\xa,2,0)}
	\draw(-.3,1.75)node[slabel]{$#1$};
	\draw(.3,.25)node[slabel]{$#2$};
}}


\def\plopen(#1)(#2){\tikz[heighttwo,xscale=.7]{%
	\tperm<(0,.3)>{#1}\tmupp<(0,1.3)>{.7}{#2}
}}
\def\plbasis(#1)(#2)(#3)(#4){\tikz[heightthree,xscale=.7]{%
	\tbbasis<(0,0)>{#1}\tperm<(0,1)>{#2}\tmupp<(0,2)>{.7}{#3}\ttbasis<(0,2.7)>{#4}
}}
\def\plclose#1(#2)(#3){\tikz[heighttwo,scale=.7]{%
	\tperm<(0,0)>{#2}\tmupp<(0,1)>{.7}{#3}\tcloser<(#1,0)>{#1}{1.7}
}}
\def\plclosesmall#1(#2)(#3){\tikz[heighttwo,scale=.5]{%
  \tikzstyle{every node}=[circle,draw,fill=white,scale=.8,inner sep=1pt]%
	\tperm<(0,1)>{#2}\tmupp<(0,2)>{.7}{#3}\tcloser<(#1,1)>{#1}{1.7}
}}

\def\tlopen(#1)(#2)(#3)(#4){\tikz[heighttwo,xscale=.7]{%
	\ttl<(0,.3)>{#1}{#2}{#3}\tmupp<(0,1.3)>{.7}{#4}}}
\def\tlbasis(#1)(#2)(#3)(#4)(#5)(#6){\tikz[heightthree,xscale=.7]{%
	\tbbasis<(0,0)>{#1}\ttl<(0,1)>{#2}{#3}{#4}\tmupp<(0,2)>{.7}{#5}\ttbasis<(0,2.7)>{#6}}}
\def\tlclose#1(#2)(#3)(#4)(#5){\tikz[heighttwo,scale=.7]{%
	\ttl<(0,0)>{#2}{#3}{#4}\tmupp<(0,1)>{.7}{#5}\tcloser<(#1,0)>{#1}{1.7}}}

\def\slopen(#1)(#2){\tikz[heighttwo,xscale=.7]{%
	\tsupp<(0,0)>{#1}{1}\tmupp<(0,1)>1{#2}}}
\def\slbasis(#1)(#2)(#3)(#4){\tikz[heightthree,xscale=.7]{%
	\tbbasis<(0,0)>{#1}\tsupp<(0,1)>{#2}{1}\tmupp<(0,2)>1{#3}\ttbasis<(0,3)>{#4}}}
\def\slclose#1(#2)(#3){\tikz[heighttwo,scale=.7]{%
	\tsupp<(0,0)>{#2}{1}\tmupp<(0,1)>{.7}{#3}\tcloser<(#1,0)>{#1}{1.7}}}


\def\mloopm(#1,#2){
	\tikz[heightoneonehalf]{
		\draw(0,-.2)..controls+(.05,.5)and+(-.3,-.4)..(1,1)node[smatrix]{$#2$}
			..controls+(.6,.8)and+(.06,.6)..(2,.5)..controls+(-.06,-.6)and+(.3,-.4)..(1,0)
			..controls+(-.3,.4)and+(-.05,-.5)..(0,1)node[smatrix]{$#1$}..controls+(.03,.3)and+(-.03,-.3)..(0,1.5);
	}}
\def\mloopmi(#1,#2,#3){
	\tikz[heightoneonehalf]{
		\draw(0,-.2)..controls+(.05,.5)and+(-.3,-.4)..(1,1)node[smatrix]{$#2$}--(1,1.7)node[svector]{$#3$};
		\draw(1,-.2)node[svector]{$#3$}..controls+(-.1,.5)and+(-.05,-.5)..(0,1)node[smatrix]{$#1$}..controls+(.03,.3)and+(-.03,-.3)..(0,1.5);
	}}

\def\mbinor(#1,#2){
	\plopen(1/2,2/1)(1/#1,2/#2)=\plopen(1/1,2/2)(1/#1,2/#2)-\tlopen(1/2)(1/2)()(1/#1,2/#2)
}
\def\mbinorrela(#1,#2){
	\tikz[heighttwo]{
		\draw(0,0)..controls+(.04,.4)and+(-.04,-.4)..node[midsmx]{$#1$}(0,2);
		\draw(1,0)..controls+(.04,.4)and+(-.04,-.4)..node[midsmx]{$#2$}(1,2);
	}
}
\def\mbinorrelb(#1){
	\tikz[heighttwo]{
		\ttl<(0,0)>{0/1}{0/1}{}
		\draw(0,1)..controls+(.04,.4)and+(-.04,-.4)..node[midsmx]{$#1$}(0,2);
		\draw(1,1)..controls+(.04,.4)and+(-.04,-.4)..(1,2);
	}
}

\def\triplev(#1,#2,#3){
	\tikz[xscale=1.2]{\hpathva(-.5,0,0)(-.3,.7,-.2)(#1)\hpathva(0,0,0)(-.3,.7,.2)(#2)\hpathva(.5,0,0)(0,1,.5)(#3)\hpatha(0,1,-.3)(-.3,.7,0)}
	=\tikz[xscale=1.2]{\hpathva(0,0,0)(.3,.7,-.2)(#2)\hpathva(.5,0,0)(.3,.7,.2)(#3)\hpathva(-.5,0,0)(0,1,-.5)(#1)\hpatha(0,1,.3)(.3,.7,0)}
	=\tikz[xscale=1.2]{\hpathva(-.5,0,0)(-.4,.7,.5)(#1)\hpathva(0,-.3,0)(0,.9,.3)(#2)\hpathva(.7,0,0)(-.4,.7,-.8)(#3)\hpatha(0,.9,-.3)(-.4,.7,0)}
	=\tikz[xscale=1.2]{\hpathva(-.5,0,0)(0,1,-.5)(#1)\hpathva(0,0,0)(0,1,0)(#2)\hpathva(.5,0,0)(0,1,.5)(#3)}
}


\def\tpermvmv(#1,#2,#3,#4,#5){\tikz[heighttwo,xscale=1.7]{
		\foreach\xa/\xb in{0/#3,.5/#4,2/#5}{\draw(\xa,0)node[svector]{$\xb$}..controls+(.05,.5)and+(-.05,-.5)..node[midsmx,pos=.7]{$#1$}(\xa,1)
			..controls+(.05,.5)and+(-.05,-.5)..(\xa,2)node[svector]{$\xb$};}
		\filldraw[fill=white](-.2,1.1)rectangle+(2.4,.4);
		\node[nlabel]at(1,1.3){$#2$};
		\foreach\xa/\xb in{1/.5,1.5/.5,1/1.9,1.5/1.9}{\node[nlabel,scale=1.5]at(\xa,\xb){.};}
}}

\def\tpermvmvca(#1,#2){\tikz[heighttwo,xscale=1.7]{
		\foreach\xa/\xb in{.5/2,2/#2}{\draw(\xa,0)node[svector]{$\xb$}..controls+(.05,.5)and+(-.05,-.5)..node[midsmx,pos=.7]{$#1$}(\xa,1)
			..controls+(.05,.5)and+(-.05,-.5)..(\xa,2)node[svector]{$\xb$};}
		\foreach\xa/\xb in{1/.5,1.5/.5,1/1.9,1.5/1.9}{\node[nlabel,scale=1.5]at(\xa,\xb){.};}
		\draw(0,0)..controls+(.05,.5)and+(-.05,-.5)..node[midsmx,pos=.35]{$#1$}(0,2)..controls+(.08,.8)and+(.1,1)..(2.5,2)
			..controls+(-.05,-.5)and+(.05,.5)..(2.5,0)..controls+(-.08,-.8)and+(-.05,-.5)..(0,0);
		\filldraw[fill=white](-.2,1.1)rectangle+(2.4,.4);
		\node[nlabel]at(1,1.3){$#2$};
	}}
\def\tpermvmvcb(#1,#2){\tikz[heighttwo,xscale=1.7]{
		\foreach\xa/\xb in{.8/3,2/#2}{\draw(\xa,0)node[svector]{$\xb$}..controls+(.05,.5)and+(-.05,-.5)..node[midsmx,pos=.7]{$#1$}(\xa,1)
			..controls+(.05,.5)and+(-.05,-.5)..(\xa,2)node[svector]{$\xb$};}
		\foreach\xa/\xb in{1.2/.5,1.6/.5,1.2/1.9,1.6/1.9}{\node[nlabel,scale=1.5]at(\xa,\xb){.};}
		\draw(0,0)..controls+(.05,.5)and+(-.05,-.5)..node[midsmx,pos=.35]{$#1$}(0,2)..controls+(.08,.8)and+(.1,1)..(3,2)
			..controls+(-.05,-.5)and+(.05,.5)..(3,0)..controls+(-.08,-.8)and+(-.05,-.5)..(0,0);
		\draw(.4,0)..controls+(.05,.5)and+(-.05,-.5)..node[midsmx,pos=.35]{$#1$}(.4,2)..controls+(.05,.5)and+(.1,1)..(2.6,2)
			..controls+(-.05,-.5)and+(.05,.5)..(2.6,0)..controls+(-.05,-.5)and+(-.05,-.5)..(.4,0);
		\filldraw[fill=white](-.2,1.1)rectangle+(2.4,.4);
		\node[nlabel]at(1,1.3){$#2$};
	}}
\def\tpermvmvcc(#1,#2){\tikz[heighttwo,xscale=1.7]{
		\foreach\xa/\xb in{1/.5,1.5/.5,1/1.9,1.5/1.9}{\node[nlabel,scale=1.5]at(\xa,\xb){.};}
		\draw(0,0)..controls+(.05,.5)and+(-.05,-.5)..node[midsmx,pos=.35]{$#1$}(0,2)..controls+(.1,1)and+(.1,1)..(3,2)
			..controls+(-.05,-.5)and+(.05,.5)..(3,0)..controls+(-.1,-1)and+(-.05,-.5)..(0,0);
		\draw(.5,0)..controls+(.05,.5)and+(-.05,-.5)..node[midsmx,pos=.35]{$#1$}(.5,2)..controls+(.07,.7)and+(.08,.8)..(2.7,2)
			..controls+(-.05,-.5)and+(.05,.5)..(2.7,0)..controls+(-.07,-.7)and+(-.05,-.5)..(.5,0);
		\draw(2,0)..controls+(.05,.5)and+(-.05,-.5)..node[midsmx,pos=.35]{$#1$}(2,2)..controls+(.03,.3)and+(.04,.4)..(2.4,2)
			..controls+(-.05,-.5)and+(.05,.5)..(2.4,0)..controls+(-.03,-.3)and+(-.04,-.4)..(2,0);
		\filldraw[fill=white](-.2,1.1)rectangle+(2.4,.4);
		\node[nlabel]at(1,1.3){$#2$};
	}}
	
\def\tchara(#1,#2){\tikz[heighttwo,xscale=1.7]{
		\foreach\xa/\xb in{1/.5,1.5/.5,1/1.9,1.5/1.9}{\node[nlabel,scale=1.5]at(\xa,\xb){.};}
		\draw(0,-.3)..controls+(.05,.5)and+(-.05,-.5)..(0,2.3);
		\draw(.5,.1)..controls+(.05,.5)and+(-.05,-.5)..node[midsmx,pos=.3]{$#1$}(.5,1.9)..controls+(.05,.5)and+(.09,.9)..(2.8,1.8)
			..controls+(-.05,-.5)and+(.05,.5)..(2.8,.2)..controls+(-.09,-.9)and+(-.05,-.5)..(.5,.1);
		\draw(2,.3)..controls+(.05,.5)and+(-.05,-.5)..node[midsmx,pos=.3]{$#1$}(2,1.5)..controls+(.05,.5)and+(.05,.5)..(2.5,1.5)
			..controls+(-.05,-.5)and+(.05,.5)..(2.5,.3)..controls+(-.05,-.5)and+(-.05,-.5)..(2,.3);
		\filldraw[fill=white](-.2,1.1)rectangle+(2.4,.4);
		\node[nlabel]at(1,1.3){$#2$};
	}}
\def\tcharb(#1,#2){\tikz[heighttwo,xscale=1.7]{
		\draw(-.5,-.3)..controls+(.05,.5)and+(-.05,-.5)..(-.5,2.3);
		\foreach\xa/\xb in{1/.5,1.5/.5,1/1.9,1.5/1.9}{\node[nlabel,scale=1.5]at(\xa,\xb){.};}
		\draw(0,.3)..controls+(.05,.5)and+(-.05,-.5)..node[midsmx,pos=.28]{$#1$}(0,1.7)..controls+(.09,.9)and+(.11,1.1)..(3.3,1.6)
			..controls+(-.05,-.5)and+(.05,.5)..(3.3,.4)..controls+(-.11,-1.1)and+(-.09,-.9)..(0,.3);
		\draw(.5,.2)..controls+(.05,.5)and+(-.05,-.5)..node[midsmx,pos=.3]{$#1$}(.5,1.8)..controls+(.05,.5)and+(.09,.9)..(2.9,1.7)
			..controls+(-.05,-.5)and+(.05,.5)..(2.9,.3)..controls+(-.09,-.9)and+(-.05,-.5)..(.5,.2);
		\draw(2,.3)..controls+(.05,.5)and+(-.05,-.5)..node[midsmx,pos=.3]{$#1$}(2,1.5)..controls+(.05,.5)and+(.05,.5)..(2.5,1.5)
			..controls+(-.05,-.5)and+(.05,.5)..(2.5,.3)..controls+(-.05,-.5)and+(-.05,-.5)..(2,.3);
		\filldraw[fill=white](-.2,1.1)rectangle+(2.4,.4);
		\node[nlabel]at(1,1.3){$#2$};
	}}
	
\def\tcharc(#1,#2){\tikz[heighttwo,xscale=1.7]{
		\foreach\xa/\xb in{1/.5,1.5/.5,1/1.9,1.5/1.9}{\node[nlabel,scale=1.5]at(\xa,\xb){.};}
		\draw(0,-.3)..controls+(.05,.5)and+(-.05,-.5)..node[midsmx,pos=.37]{$#1_1$}(0,2.3);
		\draw(.5,.1)..controls+(.05,.5)and+(-.05,-.5)..node[midsmx,pos=.3]{$#1_2$}(.5,1.9)..controls+(.05,.5)and+(.09,.9)..(2.8,1.8)
			..controls+(-.05,-.5)and+(.05,.5)..(2.8,.2)..controls+(-.09,-.9)and+(-.05,-.5)..(.5,.1);
		\draw(2,.3)..controls+(.05,.5)and+(-.05,-.5)..node[midsmx,pos=.3]{$#1_n$}(2,1.5)..controls+(.05,.5)and+(.05,.5)..(2.5,1.5)
			..controls+(-.05,-.5)and+(.05,.5)..(2.5,.3)..controls+(-.05,-.5)and+(-.05,-.5)..(2,.3);
		\filldraw[fill=white](-.2,1.1)rectangle+(2.4,.4);
		\node[nlabel]at(1,1.3){$#2$};
	}}
\def\tchard(#1,#2){\tikz[heighttwo,xscale=1.7]{
		\draw(-.5,-.3)..controls+(.05,.5)and+(-.05,-.5)..(-.5,2.3);
		\foreach\xa/\xb in{1/.5,1.5/.5,1/1.9,1.5/1.9}{\node[nlabel,scale=1.5]at(\xa,\xb){.};}
		\draw(0,.3)..controls+(.05,.5)and+(-.05,-.5)..node[midsmx,pos=.28]{$#1_1$}(0,1.7)..controls+(.09,.9)and+(.11,1.1)..(3.3,1.6)
			..controls+(-.05,-.5)and+(.05,.5)..(3.3,.4)..controls+(-.11,-1.1)and+(-.09,-.9)..(0,.3);
		\draw(.5,.2)..controls+(.05,.5)and+(-.05,-.5)..node[midsmx,pos=.3]{$#1_2$}(.5,1.8)..controls+(.05,.5)and+(.09,.9)..(2.9,1.7)
			..controls+(-.05,-.5)and+(.05,.5)..(2.9,.3)..controls+(-.09,-.9)and+(-.05,-.5)..(.5,.2);
		\draw(2,.3)..controls+(.05,.5)and+(-.05,-.5)..node[midsmx,pos=.3]{$#1_n$}(2,1.5)..controls+(.05,.5)and+(.05,.5)..(2.5,1.5)
			..controls+(-.05,-.5)and+(.05,.5)..(2.5,.3)..controls+(-.05,-.5)and+(-.05,-.5)..(2,.3);
		\filldraw[fill=white](-.2,1.1)rectangle+(2.4,.4);
		\node[nlabel]at(1,1.3){$#2$};
	}}


\def\detnn(#1){\tikz[heighttwo]{
	\cilur<(0,0)>(-.5,-.2)\cilur<(0,2)>(-.5,.2)
	\foreach\xa in{-2,-1,2}{\hpathmxa(0,0,\xa)(0,2,\xa)(#1)}
	\foreach\xa in{0,1}{\hpathph(0,0,\xa)(0,2,\xa)}
}}

\def\dethalfn(#1,#2,#3){\tikz[heighttwo,xscale=1.2]{
	\cilur<(0,0)>(-.5,-.2)\cilur<(0,2)>(-.5,.2)
	\foreach\xa in{-3,-.5}{\hpathmxa(0,0,\xa)(0,2,\xa)(#1)}
	\foreach\xa in{-2,-1.5}{\hpathph(0,0,\xa)(0,2,\xa)}\draw(-1.25,1.25)node[slabel]{$#2$};
	\foreach\xa in{.5,3}{\hpatha(0,0,\xa)(0,2,\xa)}
	\foreach\xa in{1.5,2}{\hpathph(0,0,\xa)(0,2,\xa)}\draw(1.25,1.25)node[slabel]{$#3$};	
}}

\def\detsplitn(#1,#2,#3,#4){\tikz[heighttwo,xscale=1.2]{
	\cilur<(0,0)>(-.5,-.2)\cilur<(0,2)>(-.5,.2)
	\foreach\xa in{-3,-.5}{\hpathmxa(0,0,\xa)(0,2,\xa)(#1)}
	\foreach\xa in{-2,-1.5}{\hpathph(0,0,\xa)(0,2,\xa)}\draw(-1.25,1.25)node[slabel]{$#2$};
	\foreach\xa in{.5,3}{\hpathmxa(0,0,\xa)(0,2,\xa)(#3)}
	\foreach\xa in{1.5,2}{\hpathph(0,0,\xa)(0,2,\xa)}\draw(1.25,1.25)node[slabel]{$#4$};	
}}

\def\detnnb(#1,#2){\tikz[heightthree]{
	\cilur<(0,1)>(-.5,-.2)\cilur<(0,3)>(-.5,.2)
	\foreach\xa in{-2,-1,2}{\hpathmxar(0,1,\xa)(0,3,\xa)(#1)}
	\foreach\xa in{0,1}{\hpathph(0,1,\xa)(0,3,\xa)}
	\hpathmxar(0,1,0)(0,-.5,0)(#2)
	\hpatha(0,3,0)(0,4,0)
}}

\def\vdetbv(#1,#2,#3,#4){\tikz[heightthree]{
	\cilur<(0,1)>(-.5,-.2)\cilur<(0,3)>(-.5,.2)
	\foreach\xa in{-2,-1,2}{\hpathmxa(0,3,\xa)(0,1,\xa)(#1)}
	\foreach\xa in{0,1}{\hpathph(0,1,\xa)(0,3,\xa)}
	\draw(0,-.5)node[svector]{$#4$}--node[midsmx,pos=.6]{$#2$}(0,1);
	\hpathav(0,3,0)(0,4,0)(#3)
}}

\def\adjnn(#1){\tikz[heightthree]{
	\cilur<(0,1)>(-.5,-.2)\cilur<(0,3)>(-.5,.2)
	\foreach\xa in{-2,-1,2}{\hpathmxar(0,1,\xa)(0,3,\xa)(#1)}
	\foreach\xa in{0,1}{\hpathph(0,1,\xa)(0,3,\xa)}
	\hpatha(0,.3,0)(0,1,0)
	\hpatha(0,3,0)(0,3.7,0)
}}

\def\adjhalfn(#1,#2,#3){\tikz[heightthree]{
	\cilur<(0,1)>(-.5,-.2)\cilur<(0,3)>(-.5,.2)
	\foreach\xa in{-3,-.5}{\hpathmxa(0,1,\xa)(0,3,\xa)(#1)}
	\foreach\xa in{-2,-1.5}{\hpathph(0,1,\xa)(0,3,\xa)}\draw(-1.25,1.75)node[slabel]{$#2$};
	\foreach\xa in{.5,3}{\hpatha(0,1,\xa)(0,3,\xa)}
	\foreach\xa in{1.5,2}{\hpathph(0,1,\xa)(0,3,\xa)}\draw(1.25,1.75)node[slabel]{$#3$};	
	\hpatha(0,.3,0)(0,1,0)
	\hpatha(0,3,0)(0,3.7,0)
}}

\def\vadjv(#1,#2,#3){\tikz[heightthree]{
	\cilur<(0,1)>(-.5,-.2)\cilur<(0,3)>(-.5,.2)
	\foreach\xa in{-2,-1,2}{\hpathmxa(0,3,\xa)(0,1,\xa)(#1)}
	\foreach\xa in{0,1}{\hpathph(0,1,\xa)(0,3,\xa)}
	\hpathva(0,-.2,0)(0,1,0)(#3)
	\hpathav(0,3,0)(0,4.2,0)(#2)
}}


\def\mxnodea(#1,#2,#3){\tikz[heighttwo,yscale=1.3]{
	\cilr<(0,1)>(-.8)
	\foreach\xa in{-2,-1,2}{\hpathmxa(0,1,\xa)(\xa,0,0)(#3)}
	\foreach\xa in{0,1}{\hpathph(0,1,\xa)(\xa,0,0)}
	\foreach\xa in{-2,1,2}{\hpathmxa(0,1,\xa)(\xa,2,0)(#3)}
	\foreach\xa in{-1,0}{\hpathph(0,1,\xa)(\xa,2,0)}
	\draw(-.3,1.75)node[slabel]{$#1$};
	\draw(.3,.25)node[slabel]{$#2$};
}}
\def\mxnodeb{
}
\def\mxnodec{
}
\def\mxndoed{
}


\def\upon(#1){\tikz{
	\draw(0,0)--(0,1)node[pos=.9,coordinate](A){}node[pos=.91,coordinate](B){};
	\draw[oriented](A)--(B);
	\draw[nullifier](-.15,.5)--(.15,.5)node[nlabel,right,scale=.7]{$#1$};
}}
\def\uporn(#1){\tikz[shift={(0,1)},yscale=-1]{
	\draw(0,0)--(0,1)node[pos=.9,coordinate](A){}node[pos=.91,coordinate](B){};
	\draw[oriented](A)--(B);
	\draw[nullifier](-.15,.5)--(.15,.5)node[nlabel,right,scale=.7]{$#1$};
}}
\def\upornm(#1,#2){\tikz[scale=1.3]{
	\draw(0,1)--(0,0)node[midsmx,pos=.5]{$#1$}node[pos=.9,coordinate](A){}node[pos=.91,coordinate](B){};
	\draw[oriented](A)--(B);
	\draw[nullifier](-.15,.85)--(.15,.85)node[nlabel,right,scale=.7]{$#2$};
}}
\def\uponm(#1,#2){\tikz[shift={(0,1)},yscale=-1,scale=1.3]{
	\draw(0,1)--(0,0)node[midsmx,pos=.5]{$#1$}node[pos=.25,coordinate](A){}node[pos=.26,coordinate](B){};
	\draw[oriented](A)--(B);
	\draw[nullifier](-.15,.15)--(.15,.15)node[nlabel,right,scale=.7]{$#2$};
}}

\def\detj(#1){\tikz[heighttwo,scale=1.3]{
	\cilur<(0,1.2)>(-.5,.2)
	\hpathva(0,2,0)(0,1.2,0)(#1)
	\foreach\xa in{-1.5,-.3,1.5}{\hpatha(\xa,0,0)(0,1.2,\xa)}
	\foreach\xa in{-1.1,-.7,.3,.9}{\hpathph(\xa,0,0)(0,1.2,\xa)}
}}

\def\detjj(#1){\tikz[heighttwo,scale=1.3]{
	\cilur<(0,1.2)>(-.5,.2)
	\hpathva(0,2,0)(0,1.2,0)(#1)
	\foreach\xa in{-1.5,-.3,1.5}{\hpatha(\xa,0,0)(0,1.2,\xa)}
	\foreach\xa in{-1.1,-.7,.3,.9}{\hpathph(\xa,0,0)(0,1.2,\xa)}
	\draw[nullifier](-.4,.8)--(-.1,.8)node[nlabel,right,scale=.7]{$#1$};
}}

\def\detjaj(#1,#2){\tikz[heighttwo,scale=1.3]{
	\cilur<(0,1.2)>(-.5,.2)
	\hpathva(0,2,0)(0,1.2,0)(#2)
	\foreach\xa in{-1.5,1.5}{\hpatha(\xa,0,0)(0,1.2,\xa)}
	\hpathmxa(-.3,0,0)(0,1.2,-.3)(#1)
	\foreach\xa in{-1.2,-.9,.5,1}{\hpathph(\xa,0,0)(0,1.2,\xa)}
}}

\def\detjlab(#1,#2,#3,#4){\tikz[heighttwo,scale=1.3]{
	\cilur<(0,1.2)>(-.5,.2)
	\hpathva(0,2,0)(0,1.2,0)(#3)
	\foreach\xa/\xb in{-1.5/#1,-.3/#2,1.5/#4}{\hpathva(\xa,0,0)(0,1.2,\xa)(\xb)}
	\foreach\xa in{-1.2,-.9,.5,1}{\hpathph(\xa,0,0)(0,1.2,\xa)}
}}

\def\detjlabb(#1,#2,#3,#4){\tikz[heighttwo,scale=1.3]{
	\cilur<(0,1.2)>(-.5,.2)
	\hpathva(.05,.6,0)(.05,1.2,0)(#3)
	\foreach\xa/\xb in{-1.5/#1,-.3/#2,1.5/#4}{\hpathva(\xa,0,0)(0,1.2,\xa)(\xb)}
	\foreach\xa in{-1.1,-.7,.5,1}{\hpathph(\xa,0,0)(0,1.2,\xa)}
}}


\def\pup{\:\tikz[xscale=.7]{\tupp<(0,0)>{1}{1}}\:}
\def\px{\perm(0/1,1/0)}
\def\pid{\tikz[xscale=.7]{\tupp<(0,0)>{2}{1}}}
\def\pcc{\drawtl(0/1)(0/1)()}
\def\binor{\px=\pid-\pcc}

\def\upa{\tikz{\draw(0,0)..controls+(.1,.5)and+(-.1,-.5)..(0,1)node[coordinate,pos=.55](A){}node[coordinate,pos=.56](B){};\draw[oriented](A)--(B);}}
\def\downa{\tikz{\draw(0,0)..controls+(.1,.5)and+(-.1,-.5)..(0,1)node[coordinate,pos=.55](A){}node[coordinate,pos=.56](B){};\draw[oriented](B)--(A);}}
\def\upcl{\tikz{\draw(0,0)..controls+(.1,.5)and+(-.1,-.5)..node[midedge](N){}(0,1);\cilr<(N)>(-.3)}}
\def\upcr{\tikz{\draw(0,0)..controls+(.1,.5)and+(-.1,-.5)..node[midedge](N){}(0,1);\cilr<(N)>(.3)}}
		
\def\cupp{\tikz{
	\draw(0,.7)..controls+(-.01,-.1)and+(.01,.1)..(0,.6)
		..controls+(-.04,-.4)and+(-.04,-.4)..(1,.6)..controls+(.01,.1)and+(-.01,-.1)..(1,.7);
	}}
\def\cupdot{\tikz{
	\draw(0,.7)..controls+(-.01,-.1)and+(.01,.1)..(0,.6)
		..controls+(-.04,-.4)and+(-.04,-.4)..node[dotedge](N){}(1,.6)..controls+(.01,.1)and+(-.01,-.1)..(1,.7);
	}}
\def\cupcu{\tikz{
	\draw(0,.7)..controls+(-.01,-.1)and+(.01,.1)..(0,.6)
		..controls+(-.04,-.4)and+(-.04,-.4)..node[midedge](N){}(1,.6)..controls+(.01,.1)and+(-.01,-.1)..(1,.7);
		\cilu<(N)>(.2)
	}}
\def\cupcd{\tikz{
	\draw(0,.7)..controls+(-.01,-.1)and+(.01,.1)..(0,.6)
		..controls+(-.04,-.4)and+(-.04,-.4)..node[midedge](N){}(1,.6)..controls+(.01,.1)and+(-.01,-.1)..(1,.7);
		\cilu<(N)>(-.2)
	}}
\def\cupcl{
	\tikz{\draw(0,.7)..controls+(-.01,-.1)and+(.01,.1)..(0,.5)
		..controls+(-.04,-.4)and+(-.04,-.4)..node[dotedge,pos=0](N){}(1,.5)..controls+(.01,.1)and+(-.01,-.1)..(1,.7);
		\cilr<(N)>(-.3)
	}}
\def\cupa{\tikz{\hpatha(-.5,.7,0)(0,.2,-.5)\hpathar(0,.2,.5)(.5,.7,0)}}
\def\cupar{\tikz{\hpatha(0,.2,-.5)(-.5,.7,0)\hpathar(.5,.7,0)(0,.2,.5)}}
		
\def\capp{\tikz{\draw(0,.3)..controls+(.01,.1)and+(-.01,-.1)..(0,.4)
		..controls+(.04,.4)and+(.04,.4)..(1,.4)..controls+(-.01,-.1)and+(.01,.1)..(1,.3);
	}}
\def\capdot{\tikz{\draw(0,.3)..controls+(.01,.1)and+(-.01,-.1)..(0,.4)
		..controls+(.04,.4)and+(.04,.4)..node[dotedge](N){}(1,.4)..controls+(-.01,-.1)and+(.01,.1)..(1,.3);
	}}
\def\capcu{\tikz{\draw(0,.3)..controls+(.01,.1)and+(-.01,-.1)..(0,.4)
		..controls+(.04,.4)and+(.04,.4)..node[dotedge](N){}(1,.4)..controls+(-.01,-.1)and+(.01,.1)..(1,.3);
		\cilu<(N)>(.2)
	}}
\def\capcd{\tikz{\draw(0,.3)..controls+(.01,.1)and+(-.01,-.1)..(0,.4)
		..controls+(.04,.4)and+(.04,.4)..node[dotedge](N){}(1,.4)..controls+(-.01,-.1)and+(.01,.1)..(1,.3);
		\cilu<(N)>(-.2)
	}}
\def\capa{\tikz{\hpatha(-.5,.2,0)(0,.7,-.5)\hpatha(0,.7,.5)(.5,.2,0)}}
		
\def\circb{\tikz{\draw(0,.5)..controls+(.05,.5)and+(.05,.5)..(1,.5)
		..controls+(-.05,-.5)and+(-.05,-.5)..(0,.5);
	}}
\def\circs{\tikz{\draw(0,.5)..controls+(.05,.5)and+(.05,.5)..(1,.5)
		..controls+(-.05,-.5)and+(-.05,-.5)..(0,.5);
	}}


\def\cctrl{\tikz[scale=.8,shift={(0,.1)}]{\draw(1.2,.5)..controls+(-.08,-.8)and+(-.05,-.5)..(-.2,.2)(-.2,.8)..controls+(.05,.5)and+(.08,.8)..(1.2,.5);}}
\def\cctrr{\tikz[scale=.8,shift={(0,.1)}]{\draw(-.2,.5)..controls+(-.08,-.8)and+(-.05,-.5)..(1.2,.2)(1.2,.8)..controls+(.05,.5)and+(.08,.8)..(-.2,.5);}}

\def\capcux{\tikz[yscale=.7,xscale=.8,shift={(0,-.2)}]{
		\draw(0,1)..controls+(.05,.8)and+(.05,.8)..(1,1)node[dotedge](N){};
		\cilu<(N)>(.4)
		\tperm<(0,0)>{0/1,1/0}
	}}
\def\kink{\tikz{
		\draw(0,0)..controls+(.02,.2)and+(-.02,-.2)..(0,.5)
			..controls+(.03,.3)and+(.03,.3)..(.5,.5)..controls+(-.03,-.3)and+(-.03,-.3)..(1,.5)
			..controls+(.02,.2)and+(-.02,-.2)..(1,1);}}
\def\kinka{\tikz{
		\draw(0,0)..controls+(.02,.2)and+(-.02,-.2)..(0,.5)node[coordinate,pos=.5](A){}node[coordinate,pos=.51](B){}
			..controls+(.03,.3)and+(.03,.3)..(.5,.5)..controls+(-.03,-.3)and+(-.03,-.3)..(1,.5)
			..controls+(.02,.2)and+(-.02,-.2)..(1,1)node[coordinate,pos=.5](C){}node[coordinate,pos=.51](D){};
		\draw[oriented](A)--(B);
		\draw[oriented](C)--(D);
	}}
\def\kinkdot{\tikz{
		\draw(0,0)..controls+(.02,.2)and+(-.02,-.2)..(0,.5)
			..controls+(.03,.3)and+(.03,.3)..node[dotedge]{}(.5,.5)..controls+(-.03,-.3)and+(-.03,-.3)..node[dotedge]{}(1,.5)
			..controls+(.02,.2)and+(-.02,-.2)..(1,1);}}
\def\ccdot{\tikz{
	\draw(0,0)..controls+(.1,.5)and+(.1,.5)..node[dotedge]{}(1,0);
	\draw(0,1)..controls+(-.1,-.5)and+(-.1,-.5)..node[dotedge]{}(1,1);}}
		
\def\cupsxlr{
	\tikz[heightthree,scale=.8,shift={(0,-.2)}]{\tperm<(0,2)>{0/1,1/0,3/3,4/4}\tcupp<(2,2)>{2}}
	=\tikz[heightthree,scale=.8,shift={(0,-.2)}]{\tperm<(0,2)>{0/1,1/0,3/3,4/4}\tcupp<(2,2)>{2}}}
	
\def\upcll{\tikz{
		\draw(0,0)..controls+(.02,.2)and+(-.02,-.2)..node[dotedge](M){}(0,.5)
			..controls+(.02,.2)and+(-.02,-.2)..node[dotedge](N){}(0,1);
		\cilr<(M)>(-.3)\cilr<(N)>(-.3)}
	=\tikz{\draw(0,0)..controls+(.05,.5)and+(-.05,-.5)..(0,1){};}
}

\def\triplel{
	\tikz[heighttwo,xscale=.8]{
		\draw
			(0,0)..controls+(.05,.5)and+(-.05,-.5)..(-1,1)..controls+(.05,.5)and+(-.05,-.5)..(0,2)
			(1,0)..controls+(.05,.5)and+(-.05,-.5)..(2,1)..controls+(.03,.3)and+(.03,.3)..node[dotedge](N){}(3,1)
				..controls+(-.06,-.6)and+(.06,.6)..(3,0)
			(4,0)..controls+(.06,.6)and+(-.06,-.6)..(4,1)..controls+(.1,1)and+(.1,1)..node[dotedge](M){}(1,1)
				..controls+(-.03,-.3)and+(-.03,-.3)..node[dotedge](L){}(0,1)..controls+(.05,.5)and+(-.05,-.5)..(1,2);
		\cilu<(L)>(.3)\cilu<(M)>(.3)\cilu<(N)>(.3)}
	=-\tikz[heighttwo,xscale=.8]{
		\draw
			(0,0)..controls+(.08,.8)and+(-.08,-.8)..(2,2)
			(5,0)..controls+(.08,.8)and+(-.08,-.8)..(3,2)
			(1,0)..controls+(.1,1)and+(.1,1)..node[dotedge](N){}(4,0);
		\cilu<(N)>(.4)}
}


\def\upn(#1){\tikz[xscale=.5]{
		\foreach\xa in{0,.5,2}{\draw(\xa,0)..controls+(.1,.5)and+(-.1,-.5)..(\xa,1);}
		\foreach\xa/\xb in{1/.5,1.5/.5}{\node[nlabel,scale=2]at(\xa,\xb){.};}
		\node[nlabel,scale=1.3,right]at(2,.9){$#1$};}}
\def\upcln(#1){\tikz[xscale=.7]{
		\foreach\xa in{0,.5,2}{\draw(\xa,0)..controls+(.1,.5)and+(-.1,-.5)..node[dotedge](N){}(\xa,1);\draw[ciliated](N)--+(-.4,0);}
		\foreach\xa/\xb in{1/.2,1.5/.2,1/.8,1.5/.8}{\node[nlabel,scale=1.5]at(\xa,\xb){.};}
		\node[nlabel,scale=1.3,right]at(2,.9){$#1$};}}

\def\cupcun(#1){
	\tikz[scale=.8,shift={(0,.2)}]{
		\draw(-.5,1)..controls+(-.02,-.2)and+(.02,.2)..(-.5,.7)..controls+(-.05,-.5)and+(-.05,-.5)
			..node[dotedge](M){}(.5,.7)..controls+(.02,.2)and+(-.02,0)..(.5,1);
		\draw(-2,1)..controls+(0,-.1)and+(0,.1)..(-2,.7)..controls+(-.15,-1.5)and+(-.15,-1.5)
			..node[dotedge](N){}(2,.7)..controls+(0,.1)and+(0,-.1)..(2,1)node[symlabelright]{$#1$};
		\cilu<(M)>(.3)\cilu<(N)>(.3)
		\foreach\xa/\xb in{-1.2/0,-.7/.2,1.1/0,.6/.2}{\node[nlabel,scale=2]at(\xa,\xb){.};}
	}}
		
\def\circn(#1){
	\tikz[heightthrees]{
		\ttl<(0,0)>{}{-5/-4}{}\tperm<(0,1)>{-5/-5,-4/-4,-1/-1}\ttl<(0,2)>{-5/-4}{}{}
		\draw(-1,2)..controls+(.2,1.5)and+(.2,1.5)..(-8,2)..controls+(-.05,-.5)and+(.05,.5)..(-8,1)
			..controls+(-.2,-1.5)and+(-.2,-1.5)..(-1,1);
		\filldraw[black](-4.5,1.3)rectangle+(4,.4);
		\foreach\xa/\xb in{-7/1.5,-6/1.5,-3.2/2.3,-2.2/2.4,-3.2/.7,-2.2/.6}{\node[nlabel,scale=1.5]at(\xa,\xb){.};}
		\node[nlabel,right]at(-.5,1.7){$#1$};}}

\def\slantn(#1){\tikz{
		\draw(0,0)..controls+(.04,.4)and+(-.04,-.4)..(.9,1);
		\draw(.6,0)..controls+(.04,.4)and+(-.04,-.4)..(1.5,1)node[nlabel,right]{$#1$};
		\foreach\xa/\xb in{.65/.5,.85/.5}{\node[nlabel,scale=1.5]at(\xa,\xb){.};}
}}
\def\kinkn(#1){\tikz{
		\draw(.2,-.1)..controls+(.03,.3)and+(-.03,-.3)..(0,.5)
			..controls+(.04,.4)and+(.04,.4)..(1.2,.6)..controls+(-.01,-.1)and+(-.01,-.1)..(1.5,.6)
			..controls+(.03,.3)and+(-.03,-.3)..(1.3,1.1);
		\draw(.8,-.1)..controls+(.03,.3)and+(-.03,-.3)..(.6,.4)
			..controls+(.01,.1)and+(.01,.1)..(.9,.4)..controls+(-.04,-.4)and+(-.04,-.4)..(2.1,.5)
			..controls+(.03,.3)and+(-.03,-.3)..(1.9,1.1);
		\foreach\xa/\xb in{.3/.2,.5/.2,1.6/.8,1.8/.8,.75/.6,.75/.75,1.35/.4,1.35/.25}{
			\node[nlabel,scale=1.5]at(\xa,\xb){.};}
		\node[nlabel,right]at(1.9,1.1){$#1$};}}

\def\suppn(#1){\tikz[xscale=.5]{
		\foreach\xa in{0,.5,2}{\draw(\xa,0)..controls+(.1,.5)and+(-.1,-.5)..(\xa,1);}
		\filldraw[black](-.3,.3)rectangle(2.3,.7);%
		\draw(1,.5)node[symlabel]{$#1$};
		\foreach\xa/\xb in{1/.1,1.5/.1,1/.9,1.5/.9}{\node[nlabel,scale=1.5]at(\xa,\xb){.};}}}
\def\asuppn(#1){\tikz[xscale=.5]{
		\foreach\xa in{0,.5,2}{\draw(\xa,0)..controls+(.1,.5)and+(-.1,-.5)..(\xa,1);}
		\filldraw[fill=white](-.5,.3)rectangle(2.5,.7);%
		\draw(1,.5)node[asymlabel,scale=.8]{$#1$};
		\foreach\xa/\xb in{1/.1,1.5/.1,1/.9,1.5/.9}{\node[nlabel,scale=1.5]at(\xa,\xb){.};}}}
\def\suppnbig(#1){
	\tikz[heightthrees]{
		\tperm<(0,0)>{0/0,3/3}\tperm<(0,1)>{0/0,3/3}\tperm<(0,2)>{0/0,3/3}
		\filldraw[black](-.5,1.3)rectangle+(4,.4)node[symlabelright,scale=.8]{$#1$};
		\foreach\xa/\xb in{1/.5,2/.5,1/2.5,2/2.5}{\node[nlabel,scale=1.5]at(\xa,\xb){.};}
	}}
		
\def\nodendown(#1){\tikz{
	\cilur<(0,1)>(-.5,.2)
	\foreach\xa in{-1.5,-.75,1.5}{\hpatha(\xa,0,0)(0,1,\xa)}
	\foreach\xa in{0,.75}{\hpathph(\xa,0,0)(0,1,\xa)}
	\draw(.2,.25)node[slabel]{$#1$};
}}		
		
\def\nodenup(#1){\tikz{
	\cilur<(0,0)>(-.5,-.2)
	\foreach\xa in{-1.5,-.75,1.5}{\hpatha(0,0,\xa)(\xa,1,0)}
	\foreach\xa in{0,.75}{\hpathph(0,0,\xa)(\xa,1,0)}
	\draw(.2,.65)node[slabel]{$#1$};
}}		
		
\def\asymnk(#1,#2){\tikz[heighttwo,xscale=.7]{
	\cilr<(0,1)>(-.8)
	\foreach\xa in{-2,-1,2}{\hpatha(\xa,0,0)(0,1,\xa)}
	\foreach\xa in{0,1}{\hpathph(0,1,\xa)(\xa,0,0)}
	\foreach\xa in{-2,1,2}{\hpatha(\xa,2,0)(0,1,\xa)}
	\foreach\xa in{-1,0}{\hpathph(0,1,\xa)(\xa,2,0)}
	\draw(-.3,1.75)node[slabel]{$#1$};
	\draw(.3,.25)node[slabel]{$#2$};
}}
\def\asymnkcup(#1,#2){\tikz[heighttwo]{
	\cilur<(0,1)>(-.5,.2)
	\foreach\xa in{-1.5,-1,-.2}{\hpatha(\xa,-.3,0)(0,1,\xa)}
	\foreach\xa in{.3,1.1}{\hpatha(\xa,0,0)(0,1,\xa)}
	\foreach\xa in{-.7,-.5,.6,.8}{\hpathph(\xa,0,0)(0,1,\xa)}
	\draw(.3,0)..controls+(0,-.5)and+(0,-.5)..(2,0);
	\draw(1.1,0)..controls+(0,-.2)and+(0,-.2)..(1.4,0);
	\hpatha(1.2,2,0)(1.4,0,0)\hpatha(1.8,2,0)(2,0,0)
	\hpathph(1.65,0,0)(1.65,2,0)\hpathph(1.75,0,0)(1.75,2,0)
	\draw(-.6,.25)node[slabel]{$#1$};
	\draw(.7,.25)node[slabel]{$#2$};
}}

\def\asymnkk(#1,#2,#3){\tikz[heightthree,scale=1.3]{
	\foreach\xa in{-1.5,-.75,1.5}{\hpatha(\xa,0,0)(0,.8,\xa)}
	\foreach\xa in{0,.75}{\hpathph(\xa,0,0)(0,.8,\xa)}
	\cilr<(0,.8)>(-.8)
	\foreach\xa in{-2,1,2}{\hpatha(0,2.2,\xa)(0,.8,\xa)}
	\foreach\xa in{-.75,0}{\hpathph(0,.8,\xa)(0,2.2,\xa)}
	\cilr<(0,2.2)>(-.8)
	\foreach\xa in{-1.5,-.75,1.5}{\hpatha(0,2.2,\xa)(\xa,3,0)}
	\foreach\xa in{0,.75}{\hpathph(0,2.2,\xa)(\xa,3,0)}
	\draw(.3,2.75)node[slabel]{$#1$};
	\draw(-.3,1.25)node[slabel]{$#2$};
	\draw(.3,.05)node[slabel]{$#3$};
}}
\def\asymnktwo(#1){\tikz[heightthree]{
	\foreach\xa in{-.9,.9}{\hpatha(\xa,0,0)(0,.8,\xa)}
	\cilr<(0,.8)>(-.8)
	\foreach\xa in{-2,1,2}{\hpatha(0,2.2,\xa)(0,.8,\xa)}
	\foreach\xa in{-.75,0}{\hpathph(0,.8,\xa)(0,2.2,\xa)}
	\cilr<(0,2.2)>(-.8)
	\foreach\xa in{-.9,.9}{\hpatha(0,2.2,\xa)(\xa,3,0)}
	\draw(-.3,1.25)node[slabel]{$#1$};
}}
\def\asymcapcup(#1){\tikz[scale=.8,heightthree]{
	\foreach\xa in{-1.5,-.75,1.5}{\hpatha(\xa,0,0)(0,1.2,\xa)}
	\foreach\xa in{0,.75}{\hpathph(\xa,0,0)(0,1.2,\xa)}
	\cilr<(0,1.2)>(-.8)
	\cilr<(0,1.8)>(-.8)
	\foreach\xa in{-1.5,-.75,1.5}{\hpatha(0,1.8,\xa)(\xa,3,0)}
	\foreach\xa in{0,.75}{\hpathph(0,1.8,\xa)(\xa,3,0)}
	\draw(.3,2.75)node[slabel]{$#1$};
	\draw(.3,.25)node[slabel]{$#1$};
}}
\def\acircnn(#1){\tikz[heighttwo,scale=.9]{
	\cilur<(0,0)>(-.5,-.2)\cilur<(0,2)>(-.5,.2)
	\foreach\xa in{-2,-1,2}{\hpatha(0,0,\xa)(0,2,\xa)}
	\foreach\xa in{0,1}{\hpathph(0,0,\xa)(0,2,\xa)}
	\draw(.3,1.2)node[slabel]{$#1$};
}}

\def\crossprod(#1,#2){\tikz[heighttwo]{
	\hpathva(-.7,.3,0)(0,1,-.7)(#1)
	\hpathva(.7,.3,0)(0,1,.7)(#2)
	\hpatha(0,1.5,0)(0,1,0)
}}

\def\innerprod(#1,#2){\tikz[heightone,shift={(0,1)}]{
	\hpathva(-1,0,0)(0,1,-1)(#1)
	\hpathav(0,1,1)(1,0,0)(#2)
}}

\def\crossprodB(#1,#2){\tikz[heighttwo]{
	\draw(.7,0)node[svector]{$#2$}..controls+(.03,.3)and+(-.03,-.3)..(-.7,.9);
	\draw(-.7,0)node[svector]{$#1$}..controls+(.03,.3)and+(-.03,-.3)..(.7,.9);
	\hpatha(-.7,.9,0)(0,1.4,-.7)
	\hpatha(.7,.9,0)(0,1.4,.7)
	\hpatha(0,1.8,0)(0,1.4,0)
}}


\def\upp(#1){\tikz[xscale=.7]{\tupp<(0,0)>{#1}{1}}}
\def\supp(#1){\tikz[xscale=.7]{\tsupp<(0,0)>{#1}{1}}}
\def\asupp(#1){\tikz[xscale=.7]{\tasupp<(0,0)>{#1}{1}}}


\def\perm(#1){\tikz[xscale=.7]{\tperm<(0,0)>{#1}}}
\def\lperm(#1){\tikz[xscale=.7]{\tlperm<(0,0)>{#1}}}

\def\drawtl(#1)(#2)(#3){\tikz[xscale=.7]{\ttl<(0,0)>{#1}{#2}{#3}}}


\def\tbigexc{
	\tikz[heightthree,xscale=1.3,yscale=1.1]{
		\draw(0,0)..controls+(.03,.3)and+(-.03,-.3)..(0,1)..controls+(.03,.3)and+(-.03,-.3)..(1,2)
			..controls+(.03,.5)and+(.03,.5)..node[dotedge]{}(2,2)..controls+(-.03,-.4)and+(-.03,-.4)..node[dotedge]{}(3,2)
			..controls+(.03,.3)and+(-.03,.3)..node[dotedge](M){}(3,3);
		\draw(1,0)..controls+(.03,.3)and+(-.03,-.3)..(1,1)..controls+(.03,.3)and+(-.03,-.3)..(0,2)
			..controls+(.03,.3)and+(-.03,-.3)..(0,3);
		\draw(2,0)..controls+(.03,.3)and+(-.03,-.3)..(2,1)..controls+(.03,.4)and+(.03,.4)..node[dotedge]{}(3,1)
			..controls+(-.03,-.3)and+(.03,.3)..(4,0);
		\draw(3,0)..controls+(.03,.3)and+(-.03,-.3)..(4,1)..controls+(.03,.3)and+(-.03,-.3)..node[dotedge](N){}(4,3);
		\cilr<(M)>(.4)\cilr<(N)>(-.4)
	}}
\def\tbigexccut{
	\tikz[heightthree,xscale=1.3,yscale=1.1]{
		\draw(0,0)..controls+(.03,.3)and+(-.03,-.3)..(0,1)..controls+(.03,.3)and+(-.03,-.3)..(1,2)
			..controls+(.03,.5)and+(.03,.5)..node[dotedge]{}(2,2)..controls+(-.03,-.4)and+(-.03,-.4)..node[dotedge]{}(3,2)
			..controls+(.03,.3)and+(-.03,.3)..node[dotedge](M){}(3,3);
		\draw(1,0)..controls+(.03,.3)and+(-.03,-.3)..(1,1)..controls+(.03,.3)and+(-.03,-.3)..(0,2)
			..controls+(.03,.3)and+(-.03,-.3)..(0,3);
		\draw(2,0)..controls+(.03,.3)and+(-.03,-.3)..(2,1)..controls+(.03,.4)and+(.03,.4)..node[dotedge]{}(3,1)
			..controls+(-.03,-.3)and+(.03,.3)..(4,0);
		\draw(3,0)..controls+(.03,.3)and+(-.03,-.3)..(4,1)..controls+(.03,.3)and+(-.03,-.3)..(4,2)
			..controls+(.03,.3)and+(-.03,-.3)..node[dotedge](N){}(4,3);
		\cilr<(M)>(.4)\cilr<(N)>(-.4)
		\foreach\xa in{0,1,3}{\draw[green,dotted](-.5,\xa)--(4.5,\xa);}
	}}


\def\cilex{
	\tikz[heighttwo,yscale=.7,shift={(0,1.4)}]{
	\cilr<(1,0)>(-.4)
	\foreach\xa in{154,82,10,-62,-134}{
		\draw(1,0)--+(\xa:.7);
	}	
}}
\def\cilexl{
	\tikz[heighttwo,yscale=.7,shift={(0,1.4)}]{
	\cilr<(1,0)>(-.4)
	\foreach\xa/\xb in{154/5,82/4,10/3,-62/2,-134/1}{
		\draw(1,0)--+(\xa:1)node[svector]{\xb};
	}	
}}


\def\tupp<#1>#2#3{\foreach\x in{1,...,#2}%
	\draw[shift={#1}]([shift={(-1,0)}]\x,0)..controls+(.1,.5)and+(-.1,-.5)..([shift={(-1,0)}]\x,#3);}

\def\tmupp<#1>#2#3{\foreach\xa/\xb in{#3}{\draw[shift={#1}](\xa,0)..controls+(.05,.5)and+(-.05,-.5)..node[midsmx]{\xb}(\xa,#2);}}

\def\tbbasis<#1>#2{\foreach\xa/\xb in{#2}{\draw[shift={#1}](\xa,.5)node[svector]{$\xb$}..controls+(.1,.5)and+(-.1,-.5)..(\xa,1);}}
\def\ttbasis<#1>#2{\foreach\xa/\xb in{#2}{\draw[shift={#1}](\xa,0)..controls+(.1,.5)and+(-.1,-.5)..(\xa,.5)node[svector]{$\xb$};}}

\def\tperm<#1>#2{\foreach\xa/\xb in{#2}{\draw[shift={#1}](\xa,0)..controls+(.05,.5)and+(-.05,-.5)..(\xb,1);}}
\def\tlperm<#1>#2{\foreach\xa/\xb in{#2}{\draw[shift={#1}](\xa,0)node[svector]{\xa}..controls+(.05,.5)and+(-.05,-.5)..(\xb,1)node[svector]{\xb};}}

\def\ttl<#1>#2#3#4{%
	\foreach\xa/\xb in{#2}{\draw[shift={#1}](\xa,0)..controls+(.1,.5)and+(.1,.5)..(\xb,0);}%
	\foreach\xa/\xb in{#3}{\draw[shift={#1}](\xa,1)..controls+(-.1,-.5)and+(-.1,-.5)..(\xb,1);}%
	\foreach\xa/\xb in{#4}{\draw[shift={#1}](\xa,0)..controls+(.1,.5)and+(-.1,-.5)..(\xb,1);}%
}
	
\def\tcupp<#1>#2{\foreach\x in{1,...,#2}{\draw[shift={#1}](-\x,0)..controls+(-.05,-\x)and+(-.05,-\x)..(\x,0);}}
\def\tcapn<#1>#2{\foreach\x in{1,...,#2}{\draw[shift={#1}](-\x,0)..controls+(.05,\x)and+(.05,\x)..(\x,0);}}

\def\tcloser<#1>#2#3{\tcupp<#1>{#2}\tcapn<([shift={#1}]0,#3)>{#2}\tupp<([shift={#1}]1,0)>{#2}{#3}}

\def\tsupp<#1>#2#3{\tupp<#1>{#2}{#3}\filldraw[shift={#1},black](-.5,.3)rectangle([shift={(-.5,-.3)}]#2,#3);%
			\draw([shift={#1},shift={(-.5,0)},scale=.5]#2,#3)node[symlabel]{$#2$};}
\def\tasupp<#1>#2#3{\tupp<#1>{#2}{#3}\filldraw[shift={#1},fill=white](-.5,.3)rectangle([shift={(-.5,-.3)}]#2,#3);
			\draw([shift={#1},shift={(-.5,0)},scale=.5]#2,#3)node[asymlabel]{$#2$};}		

\def\hpath(#1,#2,#3)(#4,#5,#6){\draw(#1,#2)..controls+(#3,0)and+(#6,0)..(#4,#5);}
\def\hpatha(#1,#2,#3)(#4,#5,#6){\draw(#1,#2)..controls+(#3,0)and+(#6,0)..(#4,#5)node[coordinate,pos=.55](A){}node[coordinate,pos=.56](B){};\draw[oriented](A)--(B);}
\def\hpathar(#1,#2,#3)(#4,#5,#6){\draw(#1,#2)..controls+(#3,0)and+(#6,0)..(#4,#5)node[coordinate,pos=.44](A){}node[coordinate,pos=.45](B){};\draw[oriented](A)--(B);}
\def\hpathph(#1,#2,#3)(#4,#5,#6){\draw[draw=none](#1,#2)..controls+(#3,0)and+(#6,0)..(#4,#5)node[pos=.5]{.};}
\def\hpathmx(#1,#2,#3)(#4,#5,#6)(#7){\draw(#1,#2)..controls+(#3,0)and+(#6,0)..(#4,#5)node[midsmx]{$#7$};}
\def\hpathmxa(#1,#2,#3)(#4,#5,#6)(#7){\draw(#1,#2)..controls+(#3,0)and+(#6,0)..(#4,#5)node[coordinate,pos=.75](A){}node[coordinate,pos=.76](B){}node[midsmx,pos=.45]{$#7$};\draw[oriented](A)--(B);}
\def\hpathmxar(#1,#2,#3)(#4,#5,#6)(#7){\draw(#1,#2)..controls+(#3,0)and+(#6,0)..(#4,#5)node[coordinate,pos=.25](A){}node[coordinate,pos=.24](B){}node[midsmx,pos=.55]{$#7$};\draw[oriented](A)--(B);}
\def\hpathva(#1,#2,#3)(#4,#5,#6)(#7){\draw(#1,#2)node[svector]{$#7$}..controls+(#3,0)and+(#6,0)..(#4,#5)node[coordinate,pos=.6](A){}node[coordinate,pos=.61](B){};\draw[oriented](A)--(B);}
\def\hpathav(#1,#2,#3)(#4,#5,#6)(#7){\draw(#1,#2)..controls+(#3,0)and+(#6,0)..(#4,#5)node[svector,pos=1]{$#7$}node[coordinate,pos=.5](A){}node[coordinate,pos=.51](B){};\draw[oriented](A)--(B);}
\def\hpathvav(#1,#2,#3)(#4,#5,#6)(#7,#8){\draw(#1,#2)..controls+(#3,0)and+(#6,0)..(#4,#5)node[svector,pos=0]{$#7$}node[svector,pos=1]{$#8$}node[coordinate,pos=.5](A){}node[coordinate,pos=.51](B){};\draw[oriented](A)--(B);}

\def\cilu<#1>(#2){\draw#1node[nodedot]{};\draw[ciliated]#1--+(0,#2);}
\def\cilur<#1>(#2,#3){\draw#1node[nodedot]{};\draw[ciliated]#1--+(#2,#3);}
\def\cilr<#1>(#2){\draw#1node[nodedot]{};\draw[ciliated]#1--+(#2,0);}

\def\dftn#1{\ifcase{#1}.7\or.78\or.85\or.9\or.95\or1\or1.04\or1.08\or1.11\fi}


\title{A Not-so-Characteristic Equation: the Art of Linear Algebra}
\author{Elisha Peterson}
\address{Department of Mathematical Sciences, United States Military Academy, West Point, NY 10996}
\address{{\it E-mail}: {\bf elisha.peterson@usma.edu}}
\date{\today}

\begin{abstract}
Can the cross product be generalized? Why are the trace and determinant so important in matrix theory? What do all the coefficients of the characteristic polynomial represent? This paper describes a technique for `doodling' equations from linear algebra that offers elegant solutions to all these questions. The doodles, known as trace diagrams, are graphs labeled by matrices which have a correspondence to multilinear functions. This correspondence permits computations in linear algebra to be performed using diagrams. The result is an elegant theory from which standard constructions of linear algebra such as the determinant, the trace, the adjugate matrix, Cramer's rule, and the characteristic polynomial arise naturally. Using the diagrams, it is easy to see how little structure gives rise to these various results, as they all can be `traced' back to the definition of the determinant and inner product.
\end{abstract}

\maketitle

\section{Introduction}

\arrowsoff


When I was an undergraduate, I remember asking myself: why does the cross product ``only work'' in three dimensions? And what's so special about the trace and determinant of a matrix? What is the \emph{real} reason behind the connection between the cross product and the determinant? These questions have traditional answers, but I never found them satisfying. Perhaps that's because I could not, as a visual learner, ``see'' what these things really meant.

A few years ago, I was delighted to come across the correspondences
	$$\bfu\times\bfv \leftrightarrow \crossprod(\bfu,\bfv)
	\quad\text{and}\quad
	\bfu\cdot\bfv \leftrightarrow \innerprod(\bfu,\bfv)$$
in a book by the physicist G.E. Stedman \cite{stedman90}.
Moreover, there was a way to perform \emph{rigorous} calculations using the diagrams.

The central tool in Stedman's book is the \emph{spin network}, a type of graph labeled by representations of a particular group.
Spin networks are similiar in appearance to Feynman diagrams, but are useful for different kinds of problems.
The first work with diagrams of this sort appears to be \cite{levinson56}, in which spin networks of rank 2 were used as a tool for
investigating quantized angular momenta. The name \emph{spin networks} is due to Roger Penrose, who used them to construct a discrete model of space-time \cite{penrose71}. In modern terminology, a spin network represents a graph whose edges are labeled by representations of a particular group and whose vertices are labeled by \emph{intertwiners} or maps between tensor powers of representations \cite{major99}. Predrag Cvitanovic has showed how to construct the diagrams for any Lie group, and actually found a novel classification of the classical and exceptional Lie groups using this approach \cite{cvitanovic05}.

The applications of spin networks are numerous. They are a standard tool used by physicists for studying various types of particle interactions \cite{cvitanovic05}. Their generalization to quantum groups forms the backbone of skein theory and many $3$-manifold invariants. In combinatorics, they are closely related to chromatic polynomials of graphs, and Louis Kauffman has posed the \emph{Four-Color Theorem} in terms of spin networks \cite{kauffman91}. Spin networks also play a role in geometry. They can be used to characterize the \emph{character variety} of a surface, which encodes the geometric structures which can be placed on the surface \cite{peterson06,lawtonpeterson08,sikora01}. There are indications that they may also be a powerful tool for the study of matrix invariants.

\begin{wrapfigure}{r}{0pt}\label{coolpic}\centering
\includegraphics[width=2.5in]{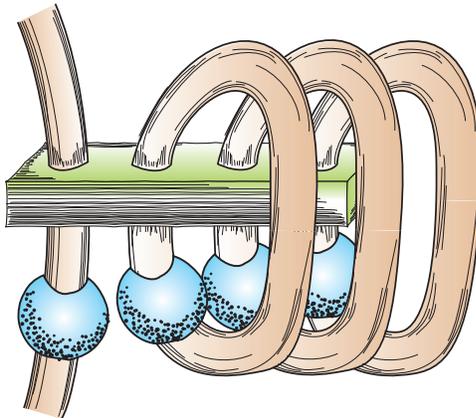}
\caption{Trace diagram which simplifies to $\det(A)I$.}
\end{wrapfigure}

Amidst all these applications, there is a surprising lack of the most basic application of the diagrams: linear algebra. In this paper, spin networks are additionally labeled with matrices and called \emph{trace diagrams} to emphasize this application. Trace diagrams provide simple depictions of traces, determinants, and other linear algebra fare in a way that is mathematically rigorous. The paper concludes with an elegant diagrammatic proof of the Cayley-Hamilton Theorem, which follows directly from the diagram in Figure \ref{coolpic}.


The emphasis of this paper is on \emph{illumination} rather than proof, and in particular on how
diagrammatic techniques have the power to both prove and explain. For this reason, several examples are
included, and more enlightening proofs are preferred. While diagrammatic methods may seem unfamiliar at first, in the end
they offer a profound insight into some of the most fundamental structures of linear algebra, such as
the determinant, the adjugate matrix, and the characteristic equation.

We hope that by the end of the paper the reader is both more comfortable with the
appearance of diagrams in mathematics and convinced that sometimes, in the words of Richard
Feynman, ``these doodles turn out to be useful.''

\section{Symmetry in Linear Algebra}\label{symmetrysection}

Perhaps the greatest contribution of diagrams is there ability to capture symmetries inherent in linear algebra.
This section reviews some of those symmetries as well as multilinear or \emph{tensor} algebra, the algebraic framework used to
interpret diagrams.

The inner product is an example of a \emph{symmetric} (or \emph{commutative}) function, since $\bfu\cdot\bfv=\bfv\cdot\bfu$. Similarly, the cross product is 
\emph{anti-symmetric} (or \emph{anti-commutative}) as $\bfu\times\bfv=-\bfv\times\bfu$. Both functions are \emph{multilinear}, since they are linear in each factor.
For example:
	$$(\bfu+\lambda\bfv)\times\bfw=\bfu\times\bfw+\lambda(\bfv\times\bfw).$$
	
Multilinear functions on vector spaces can \emph{always} be considered as functions on tensor products.
Informally, a \emph{tensor product} of two complex vector spaces consists of finite sums of vector pairs subject to the relation
	$$(\lambda\bfu,\bfv)=\lambda(\bfu,\bfv)=(\bfu,\lambda\bfv)$$
for $\lambda\in\C$. The corresponding element is usually written $\bfu\tensor\bfv$.
If $V=\C^3$, then the domain of both $\cdot$ and $\times$ can be taken to be $V\tensor V$, so that $\cdot:V\tensor V\to\C$ and $\times:V\tensor V\to V$. For a rigorous construction of tensor products, see Appendix B in \cite{fultonharris91}.

As another example, the determinant function can be written as a function $\det:V\tensor V\tensor V\to\C$, with $\bfu\tensor\bfv\tensor\bfw\mapsto\det[\bfu\spc \bfv\spc \bfw]$. The multilinearity of the determinant is usually described as being able to factor the multiplication of a matrix column by a constant outside a determinant. Since switching columns in a matrix introduces a sign on the determinant, the determinant is an \emph{anti-symmetric} function.

\section{A Taste of Trace}

This section presents \emph{3-Vector Diagrams} and \emph{Trace Diagrams} informally as heuristics to aid in calculations.
The ideas introduced here will be made rigorous in the next section.

Consider again the correspondences
	$$\bfu\times\bfv \leftrightarrow \crossprod(\bfu,\bfv)
		\quad\text{and}\quad
		\bfu\cdot\bfv \leftrightarrow \innerprod(\bfu,\bfv).$$
These diagrams are read ``bottom to top'' so that the inputs $\bfu$ and $\bfv$ occur along the bottom, and the output(s) occur along the top. In the case of the cross product, a single output strand implies a \emph{vector} output; in the case of the inner product, the absence of an input strand implies a \emph{scalar} output.

The notation permits an alternate expression of standard vector identities:		
\begin{example}
	Draw the identity
	$$(\bfu\times\bfv)\cdot(\bfw\times\bfx)=(\bfu\cdot\bfw)(\bfv\cdot\bfx)-(\bfu\cdot\bfx)(\bfv\cdot\bfw).$$
\begin{soln}
Keeping the vector inputs in the same order, the diagram is:
	$$
	\tikz[heighttwo]{
		\hpathav(-1,1,-.5)(-1.5,0,0)(\bfu)\hpathav(-1,1,.5)(-.5,0,0)(\bfv)
		\hpathva(.5,0,0)(1,1,-.5)(\bfw)\hpathva(1.5,0,0)(1,1,.5)(\bfx)
		\hpatha(-1,1,0)(0,1.7,-1)\hpatha(0,1.7,1)(1,1,0)
	}
	=
	\tikz[heighttwo]{
		\hpathav(-.5,1,-1)(-1.5,0,0)(\bfu)\hpathav(.5,1,-1)(-.5,0,0)(\bfv)
		\hpathva(.5,0,0)(-.5,1,1)(\bfw)\hpathva(1.5,0,0)(.5,1,1)(\bfx)
	}
	-
	\tikz[heighttwo]{
		\hpathav(0,1.5,-1.5)(-1.5,0,0)(\bfu)\hpathav(0,.8,-.5)(-.5,0,0)(\bfv)
		\hpathva(.5,0,0)(0,.8,.5)(\bfw)\hpathva(1.5,0,0)(0,1.5,1.5)(\bfx)
	}.
	$$
\end{soln}
\end{example}

\begin{exercise}
	What vector identity does this diagram represent?
		$$\triplev(\bfu,\bfv,\bfw)$$
	The reader is encouraged to guess the meaning of the fourth term (later described in Definition \ref{tddef}).
\end{exercise}

Now suppose that matrix and vector multiplication could also be encoded diagrammatically, according to the rules
	$$ABC \leftrightarrow \iml(ABC)=\immm(A,B,C)
	\quad\text{and}\quad
	\bfv^T A\bfw \leftrightarrow \ivmv(\bfv,A,\bfw).$$
Then matrix elements may be represented diagrammatically as well. If the standard row and column bases for $\C^n$
are represented by $\{\bse^i\}$ and $\{\bse_j\}$, respectively, then
	$$a_{ij}=\bse^i A\bse_j \leftrightarrow \ivmv(i,A,j).$$
Using this notation, trace and determinant diagrams may also be constructed.

\begin{example}
	Find a diagrammatic representation of the trace $\tr(A)$.
\begin{soln}
	$$\tr(A) \leftrightarrow \sum_{i=1}^n \ivmv(i,A,i).$$
\end{soln}
\end{example}

\begin{example}
	Find a diagrammatic representation of the determinant
	\begin{equation}\label{detformula}
		\det(A)=\sum_{\sigma\in\Sigma_n}\sgn(\sigma)a_{1\sigma(1)}a_{2\sigma(2)}\cdots a_{n\sigma(n)}.
	\end{equation}
\begin{soln}
	One approach is to introduce new notation:
	$$\det(A)\leftrightarrow\sum_{\sigma\in\Sigma_n}\sgn(\sigma)\tpermvmv(A,\sigma,1,2,n),$$
	where $\asuppn(\sigma)$ represents a permutation on the $n$ strands.
	For example, if $n=2$, then 
		$$\det(A) \leftrightarrow \tpermxel(1/A/1,2/A/2)-\tpermxel(1/A/2,2/A/1).$$
\end{soln}
\end{example}

\section{Topological Invariance: No Rules, Just Right}

What if there were a set of rules for manipulating the diagrams that was compatible with their interpretations as functions?
This section exhibits just such a correspondence, which exists \emph{provided the graphs are given a little extra structure}.
Several classical proofs become trivial in this context, as the complexity is funneled from the proof into the construction of a diagram.

The diagrams will essentially be given the structure of a graph whose vertices have degree 1 or $n$ only, and whose edges are labelled by matrices. But there are a few difficulties with ensuring a diagram's function is well-defined. First,
		\begin{equation}\label{capproblem}
			\capmr(A)=\bfu\cdot(A\bfv)=(A^T\bfu)\cdot\bfv
			\quad\text{but}\quad
			\capml(A)=(A\bfu)\cdot\bfv.
		\end{equation}
The problem here is that $\capmr(A)$ and $\capml(A)$ are indistinguishable as graphs, but represent different functions.
This problem will be resolved in Definition \ref{tddef} by requiring the graphs to be \emph{oriented}.

The second problem is
	\begin{equation}\label{crossproblem}
		\crossprod(\bfv,\bfw)=\bfv\tensor\bfw
			\quad\text{but}\quad
			\crossprodB(\bfv,\bfw)=\bfw\tensor\bfv=-\bfv\tensor\bfw.
	\end{equation}
This time, the problem is that the ordering of edges at a vertex matters, so the diagram \emph{must include a way to order the edges adjacent to a vertex}.

\arrowson
\cilson

These two extra pieces of structure are incorporated into the formal definition of a trace diagram:
\begin{defn}\label{tddef}
	An \emph{$n$-trace diagram} is an oriented graph drawn in the plane with edges labeled by $n\times n$ matrices whose vertices
	(i) have degree 1 or $n$ only,
	(ii) are sources or sinks, and
	(iii) are labeled by an ordering of edges incident to the vertex.
	
	If there are no degree 1 vertices, the trace diagram is said to be \emph{closed}.
	Otherwise, the degree 1 vertices are divided into ordered sets of \emph{inputs} and \emph{outputs}.
	These extra structures determine equivalence of the diagrams under isotopy.
	By convention, trace diagrams are typically drawn with input vertices at the bottom of the diagram and output vertices at the top.
\end{defn}
A simple way to maintain the required ordering is to draw a short mark between two edges at the vertex, called a \emph{ciliation}. For example, the ciliation on $\cilex$ implies the ordering $\cilexl$. Diagrams are unchanged under any isotopy which preserves the order of the inputs and outputs as well as the position of the ciliations. Thus
	$$\triplev(\bfu,\bfv,\bfw)$$
holds automatically since the diagrams are isotopic.

The functional interpretation of a diagram is described in the following theorem:
\begin{thm}[Fundamental Theorem of Trace Diagrams]\label{decomptheorem}
	Let $V=\C^n$.
	There is a well-defined correspondence between trace diagrams with $k$ inputs and $l$ outputs and functions $V^{\tensor k}\to V^{\tensor l}$.
	In particular, every decomposition of a trace diagram into the following basic maps gives the same function (or scalar if the diagram is closed):
	\begin{align*}
		\upa:\bfv																				&\longmapsto \bfv;\\
		\downa:\bfv^T																		&\longmapsto \bfv^T;\\
		\ima(A):\bfv																		&\longmapsto A\cdot\bfv;\\
		\cupa:1																					&\longmapsto \sum_{i=1}^n\bse^i\tensor\bse_i;\\
		\capa:\bfv\tensor\bfw^T											 		&\longmapsto \bfv\cdot\bfw;\\
		\nodendown(n):\bfv_1\tensor\cdots\tensor\bfv_n 	&\longmapsto \det[\bfv_1\spc \cdots\spc \bfv_n].
	\end{align*}
	In addition, the $n$-vertex with opposite orientation represents the determinant of several row vectors.
\end{thm}

The decomposition assumes the diagram is in `general position' relative to the vertical direction, and involves ``chopping'' the diagram up into these particular pieces. Vertical stacking of diagrams corresponds to the composition of functions. Several examples follow.

\begin{example}
	To compute $\cupar$, switch the order of the outputs:
		$$\cupar=
		\tikz{
			\hpatha(0,0,-.7)(-.7,.5,0)\hpatha(0,0,.7)(.7,.5,0);
			\draw(.7,.5)..controls+(.03,.3)and+(-.03,-.3)..(-.7,1.2);
			\draw(-.7,.5)..controls+(.03,.3)and+(-.03,-.3)..(.7,1.2);
		}:1\overset{\cupa}{\longmapsto}\sum_i\bse^i\tensor\bse_i
			 \overset{\px}{\longmapsto}\sum_i\bse_i\tensor\bse^i.$$
\end{example}

\begin{exercise}
	Show that $\tikz{\hpatha(0,0,-1)(0,1,-1)\hpatha(0,1,1)(0,0,1)}:1\mapsto\dim(V)=n$.
\end{exercise}

Thus, the basic loop is equivalent to the scalar $n$.
Here is the simplest closed diagram with a matrix, and the reason for the terminology `trace' diagram:
\begin{example}
	Show that $\tikz{\hpathmxa(0,0,-1)(0,1,-1)(A)\hpatha(0,1,1)(0,0,1)}=\tr(A)$.
\begin{soln}
	Decompose the diagram
		$\tikz{\hpathmxa(0,0,-1)(0,1,-1)(A)\hpatha(0,1,1)(0,0,1)}=\capa\circ\left(\ima(A)\spc \downa\right)\circ\cupar$.
	Then
		$$
		\tikz{\hpathmxa(0,0,-1)(0,1,-1)(A)\hpatha(0,1,1)(0,0,1)}
			:1\overset{\cupar}{\longmapsto}\sum_i \bse_i\tensor\bse^i
			\overset{\ima(A)\spc \downa}{\longmapsto}\sum_i (A\bse_i)\tensor\bse^i
			\overset{\capa}{\longmapsto}\sum_i (A\bse_i)\cdot\bse_i=\sum_i a_{ii}=\tr(A).
		$$
\end{soln}
\end{example}

\begin{exercise}
	Show that \eqref{capproblem} is no longer a problem, since $\imar(A):\bfv^T \longmapsto(A\bfv)^T.$
\end{exercise}

\begin{example}
	Compute $\kinka$.
\begin{soln}
 	$\kinka=\upa$ since the decomposition $\kinka=\capa\upa\circ\upa\cupa$ gives
  	$$
  		\kinka:\bfv
  		\overset{\upa\cupa}{\longmapsto}
  		\sum_i\bfv\tensor\bse^i\tensor\bse_i
  		\overset{\capa\upa}{\longmapsto}
  		\sum_i(\bfv\cdot\bse_i)\bse_i=\sum_i\bfv_i\bse_i=\bfv.
  	$$
\end{soln}
\end{example}

This last result was expected since both $\kinka$ and $\upa$ are equivalent forms of the same trace diagram. Moreover, this fact essentially proves Theorem \ref{decomptheorem}:

\begin{proof}[Sketch proof of Theorem \ref{decomptheorem} (adapted from \cite{peterson06})]
	The previous example shows that isotopies changing the number of local maxima and minima of the form $\cupa$ and $\capa$ do not
	change the function. Thus, a trace diagram may be assumed to be positioned with the minimum possible number of cups and caps.
	With this restriction, there is only one way to decompose a map into the component maps given in the theorem.
\end{proof}

\begin{remark}
	The correspondence between diagrams and functions established by the Fundamental Theorem of Trace Diagrams demonstrates an 
	\emph{equivalence of categories} and is well-known in literature on spin networks.
	It is usually proven by showing that the algebra of diagrams modulo specific relations is equivalent to the algebra of functions
	modulo their relations \cite{stedman90}.
\end{remark}

\begin{remark}
	In particular cases, the diagrams can be simplified somewhat.
	\begin{itemize}
	\item If $n$ is odd, then the determinant is \emph{cyclically invariant}:
					$$\det[\bfv_1\spc \cdots\spc \bfv_n]=\det[\bfv_2\spc \cdots\spc \bfv_n\spc \bfv_1],$$
				so the ciliation is unnecessary; the cyclic orientation implied by the drawing of a diagram in the plane is sufficient.
	\item	For $2$-trace diagrams, the orientation is unnecessary,
				and frequently the cap $\capp$ is defined to be $\bfv\tensor\bfw \mapsto i\cdot\det[\bfv\spc \bfw]$ rather than the inner product.
				With the factor of $i$, there is no need to keep track of ciliations, and the diagrams are simply collections of arcs and loops
				labelled by matrices \cite{carter95,kauffman91}.
	\end{itemize}
\end{remark}

\section{Simplifying Diagrammatic Calculations}

Theorem \ref{decomptheorem} says that a function's diagram may be computed by transforming the diagram into some sort of ``standard form.'' In practice, it is better to have a working knowledge of the functions of the few basic diagrams described next.
This is abundantly clear in the following computation:

\begin{example}
	Show that the diagram
	$\tikz[heighttwo,scale=.8]{
		\hpatha(0,0,0)(0,.7,0)\hpatha(0,1.7,0)(0,2.4,0)\hpatha(0,1.7,-1)(0,.7,-1)\hpatha(0,1.7,1)(0,.7,1)
		\cilur<(0,.7)>(-.1,-.5)\cilur<(0,1.7)>(.1,-.5)
	}$
	is the function $\bfv\mapsto-2\bfv$.
\begin{soln}
	Arrange the diagram so it decomposes into the basic component maps. Then:
	\begin{align*}
		\tikz[heighttwo,yscale=.65]{
			\hpatha(-.5,-.5,0)(-.5,1,0)
			\hpatha(-.5,1,0)(0,2,-.5)\hpatha(0,1,0)(0,2,0)\hpatha(.5,1,0)(0,2,.5)
			\hpatha(1,0,-1)(0,1,0)\hpatha(1,.5,-.5)(.5,1,0)
			\hpatha(2,1,0)(1,0,1)\hpatha(1.5,1,0)(1,.5,.5)
			\hpatha(2,2,-.5)(1.5,1,0)\hpatha(2,2,0)(2,1,0)\hpatha(2,2,.5)(2.5,1,0)
			\hpatha(2.5,1,0)(3,.5,-.5)\hpatha(3,.5,.5)(3.5,1,0)
			\hpatha(3.5,1,0)(3.5,2.5,0)
			\cilr<(0,2)>(-.5)\cilr<(2,2)>(-.5)
		}:v&\mapsto\sum_{i,j,k}\bfv\tensor\bse_i\tensor\bse_j\tensor\bse^j\tensor\bse^i\tensor\bse^k\tensor\bse_k\\
			&=\sum_{i\neq j,j\neq k,i\neq k}\det[\bfv\spc \bse_i\spc \bse_j]\det[\bse^j\spc \bse^i\spc \bse^k]^T \bse_k.
	\end{align*}
	But
		\begin{align*}
			\det[\bfv\spc \bse_i\spc \bse_j]\det[\bse^j\spc \bse^i\spc \bse^k]^T
			=\det
			\begin{bmatrix}
				\bfv\cdot\bse_j & \bfv\cdot\bse_i & \bfv\cdot\bse_k\\
				0 & 1 & 0\\
				1 & 0 & 0
			\end{bmatrix}
			&=-\bfv\cdot\bse_k=-v_k.
		\end{align*}
	Therefore,
		$$\tikz[heighttwo,scale=.8]{
				\hpatha(0,0,0)(0,.7,0)\hpatha(0,1.7,0)(0,2.4,0)\hpatha(0,1.7,-1)(0,.7,-1)\hpatha(0,1.7,1)(0,.7,1)
				\cilr<(0,.7)>(-.5)\cilr<(0,1.7)>(-.5)
			}:\bfv\mapsto\sum_{i\neq j,j\neq k,i\neq k}(-v_k)\bse_k = -2\sum_k v_k \bse_k = -2\bfv.$$
\end{soln}
\end{example}

Fortunately, there is an easier approach. The remainder of this section gives several results that make such technical calculations unnecessary. 
Several of the proofs are deferred to the appendix.

\begin{prop}\label{nodeantisymmetry}
	The $n$-vertices of trace diagrams are antisymmetric, meaning a sign is introduced if either (i) two inputs at a node are switched or (ii)
	a ciliation is moved across an edge. Alternately, if any collection of inputs at an $n$-vertex are linearly dependent,
	the diagram evaluates to zero.
\begin{proof}
	These facts are restatements of standard facts regarding the determinant: switching two columns introduces a sign, while the determinant
	of a square matrix with linearly dependent rows or columns is zero.
\end{proof}
\end{prop}

\begin{prop}[Matrix Invariance]\label{mxinvthm}
	Matrices may be `cancelled' at any node with the addition of a determinant factor. In particular,
		$$\mxnodea(,,A)=\det(A)\asymnk(,)
		\quad\text{and}\quad
		\asymnkmx(,,A)=\det(A)\asymnkmxB(,,\bar A),$$
	where $\bar A=A^{-1}$.
\begin{proof}
	Topological invariance implies that only one case need be checked:
	$$
		\nodendownmx(n,A): \bfv_1\dtensor\bfv_n \mapsto \det[A\bfv_1\spc \cdots\spc A\bfv_n]=\det(A)\det[\bfv_1\spc \cdots\spc \bfv_n].
	$$
	The second relation follows from the first.
\end{proof}
\end{prop}

The next two propositions, whose proofs are found in the appendix, give the functions at a diagram's nodes:
\begin{prop}\label{standardnode}
	The following diagram is a ``complemental antisymmetrizer'':
		\begin{equation}\label{cantisym}
			\asymnk(n-k,k):\bse_{\alpha_1}\tensor\cdots\tensor\bse_{\alpha_k} \longmapsto
			\sum_{\sigma\in\Sigma_{n-k}}\sgn(\alpha|\overset\leftarrow\sigma)\bse^{\sigma(k+1)}\tensor\cdots\tensor\bse^{\sigma(n)},
		\end{equation}
	where the sum is over permutations on the complement of $\{\alpha_1,\ldots,\alpha_k\}$ and
		$$\sgn(\alpha|\overset\leftarrow\sigma)
				=\sgn\left(\alpha_1\spc \cdots\spc \alpha_k\sigma(n)\spc \cdots\spc \sigma(k+1)\right).$$
	One important special case is the \emph{codeterminant} map
		\begin{equation}
			\nodenup(n):1\mapsto\sum_{\sigma\in\Sigma_n}\sgn(\overset\leftarrow\sigma)\bse_{\sigma(1)}\tensor\cdots\tensor\bse_{\sigma(n)}.
		\end{equation}
\end{prop}

Note that $\sgn(\overset\leftarrow\sigma)=(-1)^{\lfloor\frac{n}{2}\rfloor}\sgn(\sigma)$ since $\lfloor\frac{n}{2}\rfloor$ swaps are required to reverse the order of a permutation on $n$ elements.

\begin{remark}
	Proposition \ref{standardnode} shows that there is a way to define a ``cross product'' in \emph{any} dimension, although $n-1$ inputs
	are required. In particular, $\crossprod(\bfu,\bfv)$ extends to the family of diagrams
	$$\detjlab(\bfa_1,\bfa_i,,\bfa_{n-1}).$$
\end{remark}

\begin{prop}\label{asymcompare}
	Let $\asuppn(k)$ denote the antisymmetrizer on $k$ vertices defined by
		$$\asuppn(k):\bfa_1\dtensor\bfa_k\longmapsto\sum_{\sigma\in\Sigma_k}\sgn(\sigma)a_{\sigma(1)}\dtensor a_{\sigma(k)}.$$
	If $k>n$, then $\asuppn(k)=0$. Otherwise, for $0\leq k\leq n$,
		$$\asuppn(k)=\frac{(-1)^{\lfloor\frac{n}{2}\rfloor}}{(n-k)!}\asymnkk(k,n-k,k).$$
	In particular cases:
		\begin{equation}\label{asymcapcup}
			\asymcapcup(n)=(-1)^{\lfloor\frac{n}{2}\rfloor}\asuppn(n)
			\qquad
			\acircnn(n)=(-1)^{\lfloor\frac{n}{2}\rfloor}n!
			\qquad
			\detnn(A)=(-1)^{\lfloor\frac{n}{2}\rfloor} n! \det(A).
		\end{equation}
\end{prop}

\section{The Elegant Adjugate and Clandestine Cramer}

We now turn to the main results of this paper, the diagrammatic proofs of the adjugate formula,
Cramer's Rule, and the Cayley-Hamilton Theorem. For each, we will first establish the diagrammatic result,
and later show that the diagrammatic relation is equivalent to its standard expression via linear algebra.

The diagrammatic version of the adjugate formula $\adj(A)\cdot A=\det(A) I$ is
\begin{prop}[Diagrammatic Adjugate Formula]\label{adjprop}
	\begin{equation}\label{adjdpf}
		\detnnb(A,A)=(-1)^{\lfloor\frac{n}{2}\rfloor}(n-1)!\det(A)\upa.
	\end{equation}
\begin{proof}
	Use Propositions \ref{mxinvthm} and \ref{asymcompare}.
\end{proof}
\end{prop}

Cramer's Rule is actually `hiding' in this diagram.
Recall that if $A=[\bfa_1\spc \bfa_2\spc \cdots\spc \bfa_n]$, then \emph{Cramer's Rule} states that the elements of the solution $\bfx$ are given by
	$$x_j=\frac{\det(A_j)}{\det(A)}=\frac{\det[\bfa_1\spc \cdots\spc \bfa_{j-1}\spc \bfb\spc \bfa_{j+1}\spc \cdots\spc \bfa_n]}{\det(A)},$$
where $A_j$ is, as shown, the matrix obtained from $A$ by replacing the $j$th column with the vector $\bfb$.

\begin{prop}[Diagrammatic Cramer's Rule]\label{cramerdthm}
	Suppose that the columns of $A_j$ are identical to the columns of $A$, except that the $j$th column of $A_j$ is $\bfx$. Then
	\begin{equation}\label{cramerdpf}
		\vdetbv(A_j,A_j,j,j)=(-1)^{\lfloor\frac{n}{2}\rfloor}(n-1)!\det(A)\ivv(j,\bfx).
	\end{equation}
\begin{proof}
	By Propositions \ref{mxinvthm} and \ref{asymcompare},
	\begin{equation}\label{cramerproof}
		\det(A_j)=\frac{(-1)^{\lfloor\frac{n}{2}\rfloor}}{(n-1)!}\vdetbv(A_j,A_j,j,j)
		=\frac{(-1)^{\lfloor\frac{n}{2}\rfloor}}{(n-1)!}\vdetbv(A,A,j,\bfx)
		=\det(A)\ivv(j,\bfx)=\det(A)x_j.
	\end{equation}
	The second step follows from the fact that $A_j \bse_j=\bfb=A\bfx$ and the following lemma:
	\begin{lemma}\label{crossoutlemma}
		Let $\upon(j)$ represent the linear function mapping $\bse_j\mapsto 0$ and $\bse_i\mapsto 1$ for $i\neq j$.
		If the matrix $A_j$ is any matrix obtained by replacing the $j$th column of $A$ by an arbitrary vector, then
		  $$\uponm(A,j)=\uponm(A_j,j)\quad\text{and}\quad\detjaj(A,j)=\detjaj(A_j,j).$$
	\begin{proof}
		The first statement holds because $A\bse_i=A_j\bse_i$ unless $i=j$, in which case both diagams evaluate to zero.
		For the second statement, note that if two strands adjacent to a single node are labeled with the same vector, then the
		diagram evaluates to zero. Consequently,
			$$\detj(j)=\detjj(j).$$
		Combined with the first statement, this proves the second.
	\end{proof}
	\end{lemma}
\end{proof}	
\end{prop}

It should be clear from \ref{cramerproof} that Proposition \eqref{cramerdpf} implies establish Cramer's Rule, but it may not be clear where the adjugate matrix shows up in Proposition \ref{adjprop}. In traditional texts, the adjugate matrix is constructed by (i) building a matrix of cofactors, (ii) applying sign changes along a checkerboard pattern, and (iii) transposing the result.

In contrast, the diagram $\adj(A)$ is quite simple:
\begin{prop}
	The matrix elements of $\adj(A)$ may be expressed as
	$$(\adj(A))_{ji}=\frac{(-1)^{\lfloor\frac{n}{2}\rfloor}}{(n-1)!}\vadjv(A,j,i),$$
\begin{proof}
The matrix element of the diagram must somehow encode all the traditional steps required for computing the adjugate. But how?
First, `crossing out' occurs when the basis elements $\bse_i$ and $\bse_j$ are placed adjacent to the nodes, sign changes are encoded in the orientation of the node, and the transpose comes into play because the matrices in the diagram are along downward-oriented strands.

Formally, note that the signed $(i,j)$-cofactor may be expressed as the determinant of the matrix
		$$A_j=[\bfa_1\spc \cdots\spc \bfa_{j-1}\spc \bse_i\spc \bfa_{j+1}\spc \cdots\spc \bfa_n]
			=\begin{bmatrix}
					a_{11} & \cdots & 0 		 & \cdots & a_{1n}\\
					\vdots & \ddots & \vdots & \ddots & \vdots\\
					a_{i1} & \cdots & 1		   & \cdots & a_{in}\\
					\vdots & \ddots & \vdots & \ddots & \vdots\\
					a_{n1} & \cdots & 0 		 & \cdots & a_{nn}
			\end{bmatrix}
		,$$
in which the $j$th column of $A$ is replaced by $\bse_i$.
By definition, $(\adj(A))_{ji}=\det(A_j)$.

The following equation, which is remarkably similar to \eqref{cramerproof}, shows how to find the adjugate matrix:
		\begin{equation}\label{adjproof}
	  	(\adj(A))_{ji}=\det(A_j)=\frac{(-1)^{\lfloor\frac{n}{2}\rfloor}}{(n-1)!}\vdetbv(A_j,A_j,j,j)
	  	=\frac{(-1)^{\lfloor\frac{n}{2}\rfloor}}{(n-1)!}\vadjv(A,j,i).
	  \end{equation}
This establishes the result.
\end{proof}
\end{prop}

\section{The Not-So-Characteristic Equation}

Recall that $p(\lambda)=\det(A-\lambda I)$ is a degree $n$ polynomial in $\lambda$, called the \emph{characteristic polynomial} of $A$.
By the fundamental theorem of algebra, it must have $n$ real or complex roots, counted with multiplicity, which are the \emph{eigenvalues} of $A$ leading to the equation
$A\bfv=\lambda\bfv$ for some \emph{eigenvector} $\bfv$. The \emph{Cayley-Hamilton Theorem} says that a matrix satisfies its own characteristic equation $p(A)=0$. For example, in
the case $A=\tmxt{2&3}{4&5}$, the characteristic polynomial is%
	$$\det(A-\lambda I)=\lambda^2-(\tr A)\lambda+\det(A)=\lambda^2-7\lambda-2,$$ 
and consequently
	$$\tmx{2&3}{4&5}^2-7\tmx{2&3}{4&5}-2\tmx{1&0}{0&1}=\tmx{0&0}{0&0}.$$
	
Our final result is actually the most trivial, an immediate consequence of Proposition \ref{asymcompare}:
\begin{thm}[Diagrammatic Cayley-Hamilton Theorem]\label{diagcayley}
	If $\asuppn(n+1)$ represents the anti-symmetrizer on $n+1$ vectors, then
	\begin{equation}\label{cayleydpf}
		\tchara(A,n+1)=0.
	\end{equation}
\end{thm}
But how does this \emph{not-so-characteristic} version of the characteristic equation relate to the formula $\det(A-\lambda I)$?

\begin{prop}\label{propcharcoeff}
	When the anti-symmetrizer $\asuppn(n+1)$ in \eqref{cayleydpf} is expanded, the coefficients of $A^i$ are equal to $n!$ times the
	coefficients of $\lambda^i$ in the characteristic polynomial $\det(A-\lambda I)=0$.
\end{prop}

As a first step, here is the determinant of a sum of matrices:
\begin{lemma}\label{detsum}
  Given $A,B\in\mnn$, the determinant sum $\det(A+B)$ is expressed diagrammatically as
  	$$\det(A+B)=\frac{(-1)^{\lfloor\frac{n}{2}\rfloor}}{n!}\sum_{i=0}^n\pmx{n}{i}\detsplitn(A,n-i,B,i).$$
\begin{proof}
	Replace $A$ in $\det(A)=\frac{(-1)^{\lfloor\frac{n}{2}\rfloor}}{n!}\detnn(A)$ with $A+B$.
	The result consists of $2^n$ diagrams, each of which has the form $\detnn(?)$ where $?$ is either $A$
	or $B$. The matrices in each summand may be reordered without introducing any signs, since any `switch' at the lower node must also be
	made at the upper node. So the terms may be grouped by the number of $B$'s in each. Since there are $\pmx{n}{i}$ with $i$ $B$ strands, the
	result follows.
\end{proof}
\end{lemma}
Consequently, the characteristic polynomial is 
	\begin{equation}\label{charcoeff}
		\det(A-\lambda I)=\frac{(-1)^{\lfloor\frac{n}{2}\rfloor}}{n!}\sum_{i=0}^n\left((-1)^i\pmx{n}{i}\dethalfn(A,n-i,i)\right)\lambda^i.
	\end{equation}
This means that the coefficients of the characteristic polynomial are, up to a constant factor, the $n+1$ ``simplest''
diagrams with nodes.

The next lemma provides a combinatorial decomposition of the anti-symmetrizer in \eqref{cayleydpf}:
\begin{lemma}\label{asymsum}
	For any $k$ with $0\leq k\leq n$,
		\begin{equation}\label{asymsumformula}
			\tchara(A,k+1)=\sum_{i=0}^k\frac{(-1)^i k!}{(k-i)!}\tpermvmvcc(A,k-i)\iml(A^i).
		\end{equation}
\begin{proof}
	Let $\sigma\in\Sigma_{k+1}$ be a permutation, and decompose it $\sigma=\tau\nu$,
	where $\tau$ is the cycle containing the first element and $\nu$ contains the remaining cycles.
	If $|\tau|=i+1$ is fixed, then there are $\frac{k!}{(k-i)!}$ choices for $\tau$ and therefore $k!$ total choices for
	each value of $i$. Since $\sgn(\tau)=(-1)^i$ and $\sgn(\nu)$ is incorporated into $\asuppn(k-i)$, the coefficient of $A^i$ is
	 $(-1)^i\frac{k!}{(k-i)!}$.
\end{proof}
\end{lemma}

\begin{proof}[Proof of Proposition \ref{propcharcoeff}]
	As indicated in \eqref{charcoeff}, the coefficient of $\lambda^i$ in the characteristic polynomial is
		\begin{equation}\label{coeff1}
			c_i=\frac{(-1)^{i+\lfloor\frac{n}{2}\rfloor}}{i!(n-i)!}\dethalfn(A,n-i,i).
		\end{equation}
	On the other hand, letting $k=n$ in Lemma \ref{asymsum} and applying Proposition \ref{asymcompare} shows that
		\begin{align*}
			\tchara(A,n+1)&=\sum_{i=0}^n\frac{(-1)^i n!}{(n-i)!}\tpermvmvcc(A,n-i)\iml(A^i)\\
			&=\sum_{i=0}^n\frac{(-1)^i(-1)^{\lfloor\frac{n}{2}\rfloor}n!}{(n-i)!i!}\dethalfn(A,n-i,i)A^i.
		\end{align*}
	The coefficient in the last sum is $n! c_i$, so Equation \eqref{cayleydpf} is $n!$ times that given by the
	Cayley-Hamilton Theorem.
\end{proof}

\section{Closing Thoughts}

The results contained in this paper are just a start. It is an easy exercise to flip through a linear algebra textbook and find more results than can be expressed diagrammatically. Conversely, every diagrammatic relation has a linear algebra interpretation, so they are quite good at ``generating'' new formulas. 

As one final example, trace diagrams are remarkably good at capturing the concept of a matrix minor. Using this fact, Steven Morse has shown \cite{morse08} that both the Jacobi Determinant Theorem and Charles Dodsgon's \emph{condensation method} for calculating determinants have simple diagrammatic proofs.
The Jacobi Determinant Theorem is
		$$
		\tikz[yscale=.8,shift={(0,-1.5)}]{
			\hpatha(-1,0,0)(0,1,-1)\hpatha(1,0,0)(0,1,1)
			\hpathph(-.3,0,0)(0,1,-.3)\hpathph(.3,0,0)(0,1,.3)
			\hpatha(-2,2,0)(0,1,-2)\hpatha(2,2,0)(0,1,2)
			\foreach\xa/\xb in{-2/-1,-2/1,2/-1,2/1}{\hpathmxa(\xa,2,\xb)(\xa,3,\xb)(A)}
			\foreach\xa/\xb in{-2/-.2,-2/.2,2/-.2,2/.2,0/-.7,0/-.4,0/.4,0/.7}{\hpathph(\xa,2,\xb)(\xa,3,\xb)}
			\hpatha(0,4,-2)(-2,3,0)\hpatha(0,4,2)(2,3,0)
			\hpathph(-.3,5,0)(0,4,-.3)\hpathph(.3,5,0)(0,4,.3)
			\hpatha(0,4,-1)(-1,5,0)\hpatha(0,4,1)(1,5,0)
			\cilr<(0,1)>(-.7)
			\cilr<(0,4)>(-.7)
			\draw(0,.3)node[slabel]{$n-k$};
			\draw(0,2.8)node[slabel]{$k$};
			\draw(0,4.7)node[slabel]{$n-k$};
		}
		=
		\tikz[heightthree,scale=1.3]{
			\foreach\xa in{-1.5,-.75,1.5}{\hpatha(\xa,0,0)(0,.8,\xa)}
			\foreach\xa in{0,.75}{\hpathph(\xa,0,0)(0,.8,\xa)}
			\cilr<(0,.8)>(-.8)
			\foreach\xa in{-2,1,2}{\hpathmxa(0,2.2,\xa)(0,.8,\xa)(A)}
			\foreach\xa in{-.75,0}{\hpathph(0,.8,\xa)(0,2.2,\xa)}
			\cilr<(0,2.2)>(-.8)
			\foreach\xa in{-1.5,-.75,1.5}{\hpatha(0,2.2,\xa)(\xa,3,0)}
			\foreach\xa in{0,.75}{\hpathph(0,2.2,\xa)(\xa,3,0)}
			\draw(.3,2.75)node[slabel]{$n-k$};
			\draw(-.3,1.25)node[slabel]{$k$};
			\draw(.3,.05)node[slabel]{$n-k$};
		}
		$$
	and the condensation technique arises from the special case $k=n-2$.

Trace diagrams are a powerful technique for generating trace identities, which suggests that they may in the near future be
used to solve problems in invariant theory that are intractable with classic techniques.
Indeed, for $n$ matrices $\{A_1,\ldots,A_n\}$ it is an immediate consequence of Proposition \ref{asymcompare} that
	$$\tcharc(A,n)=\tchard(A,n).$$
This equation, which generalizes the characteristic equation, is sometimes called a \emph{polarization} of the characteristic polynomial, and is the source of all trace identities for certain matrix groups.

\appendix
\section{Proofs of Node Identities}

\begin{prop}[Proposition \ref{standardnode}]
		$$
			\asymnk(n-k,k):\bse_{\alpha_1}\tensor\cdots\tensor\bse_{\alpha_k} \longmapsto
			\sum_{\sigma\in\Sigma_{n-k}}\sgn(\alpha|\overset\leftarrow\sigma)\bse^{\sigma(k+1)}\tensor\cdots\tensor\bse^{\sigma(n)},
		$$
	where the sum is over permutations on the complement of $\{\alpha_1,\ldots,\alpha_k\}$ and
		$$\sgn(\alpha|\overset\leftarrow\sigma)
				=\sgn\left(\alpha_1\spc \cdots\spc \alpha_k\sigma(n)\spc \cdots\spc \sigma(k+1)\right).$$
\begin{proof}
	The function is computed by redrawing the diagram
		$$\asymnkcup(k,) = \asymnk(n-k,k).$$
	The cups transform an input $\bse_{\alpha_1}\dtensor\bse_{\alpha_k}$ to 
		$$\sum_{\sigma\in\Sigma_{n-k}}\bse_{\alpha_1}\dtensor\bse_{\alpha_k}
			\tensor\left(\bse_{\sigma(n)}\dtensor\bse_{\sigma(k+1)}\right)
			\tensor\left(\bse^{\sigma(k+1)}\tensor\cdots\tensor\bse^{\sigma(n)}\right).$$
	The summation is restricted to the complement of $\{\alpha_i\}$ since the determinant of any matrix with repeated columns is zero.
	Hence, it may be rewritten
		$$\sum_{\sigma\in\Sigma_{n-k}}\det[\bse_{\alpha_1}\spc \cdots\spc \bse_{\alpha_k}\spc \bse_{\sigma(n)}\spc \cdots\spc \bse_{\sigma(k+1)}]
			\bse^{\sigma(k+1)}\dtensor\bse^{\sigma(n)},$$
	and the determinant evaluates to the sign defined above.
\end{proof}
\end{prop}

\begin{prop}[Proposition \ref{asymcompare}]	
	If $k>n$, then $\asuppn(k)=0$. Otherwise, for $0\leq k\leq n$,
		$$\asuppn(k)=\frac{(-1)^{\lfloor\frac{n}{2}\rfloor}}{(n-k)!}\asymnkk(k,n-k,k).$$
\begin{proof}
	The fact that $\asuppn(k)=0$ for $k>n$ follows from	Proposition \ref{nodeantisymmetry}.

	For $k\leq n$, the function is computed by applying Proposition \ref{standardnode} twice.
	The image of $\bse_{\alpha_1}\tensor\cdots\tensor\bse_{\alpha_k}$ is
		$$
			\sum_{\tau\in\Sigma_k}\sum_{\sigma\in\Sigma_{n-k}}
			\sgn(\Id|\overset\leftarrow\sigma)\sgn(\sigma|\overset\leftarrow\tau)
				\bse_{\tau(1)}\dtensor\bse_{\tau(k)}
		$$
	Reversing the permutation in the $\sgn(\sigma|\overset\leftarrow\tau)$ term gives:
		\begin{align*}
			\sgn(\Id|\overset\leftarrow\sigma)\sgn(\sigma|\overset\leftarrow\tau)
				&=(-1)^{\lfloor\frac{n}{2}\rfloor}\sgn(\Id|\overset\leftarrow\sigma)\sgn(\tau|\overset\leftarrow\sigma)\\
				&=(-1)^{\lfloor\frac{n}{2}\rfloor}\sgn(\overset\leftarrow\sigma)^2\sgn(\beta)^2\sgn(\tau)\\
				&=(-1)^{\lfloor\frac{n}{2}\rfloor}\sgn(\tau).
		\end{align*}
	Here $\beta$ is a permutation that encodes the ordering of $\{\alpha_i\}$ and its complement, which is assumed
	to be consistent throughout the computation.
	Therefore,
		\begin{align*}
		\bse_{\alpha_1}\tensor\cdots\tensor\bse_{\alpha_k}
			&\longmapsto
			\sum_{\tau\in\Sigma_k}\sum_{\sigma\in\Sigma_{n-k}}(-1)^{\lfloor\frac{n}{2}\rfloor}\sgn(\tau)
				\bse_{\tau(1)}\dtensor\bse_{\tau(k)}\\
			&=(-1)^{\lfloor\frac{n}{2}\rfloor}(n-k)!\sum_{\tau\in\Sigma_k}\sgn(\tau)\bse_{\tau(1)}\dtensor\bse_{\tau(k)}.
			\qedhere
		\end{align*}
\end{proof}
\end{prop}

\bibliographystyle{plain}
\bibliography{elisha}

\end{document}